\documentclass[12pt,reqno,a4paper]{amsart}
\usepackage[T1]{fontenc}
\usepackage{hyperref}
\usepackage{enumerate}
\usepackage[left=0.65in, right=0.65in, top= 1in, bottom= 1.55in ]{geometry}
\usepackage{amsmath,amssymb,amsthm,amsfonts}
\usepackage{graphicx}
\usepackage{tikz}
\usetikzlibrary{graphs,graphs.standard}
\usetikzlibrary{shapes}
\usetikzlibrary{plotmarks}
\usetikzlibrary{arrows}
\usetikzlibrary{positioning}
\usepackage{calrsfs}
\theoremstyle{definition}
\newtheorem{defn}{Definition}[section]
\newtheorem{fact}[defn]{Fact}
\newtheorem{prop}[defn]{Proposition}
\newtheorem{thm}[defn]{Theorem}

\newtheorem{lem}[defn]{Lemma}
\newtheorem{question}[defn]{Question}
\newtheorem{remark}[defn]{Remark}

\newtheorem{claim}[defn]{claim}

\title[Rainbow Ramsey theorem and CRT without AC]{On the Rainbow Ramsey theorem, and the Canonical Ramsey Theorem for pairs without AC}
\author{Amitayu Banerjee}
\address{Alfr\'ed R\'enyi Institute of Mathematics, Reáltanoda utca 13-15, 1053, Budapest, Hungary}
\email[Corresponding author]{banerjee.amitayu@gmail.com}

\author{Alexa Gopaulsingh}
\address{E\"otv\"os Lor\'and University, Department of Logic, M\'{u}zeum krt. 4, 1088, Budapest, Hungary}
\email{alexa279e@gmail.com}

\author{Zal\'{a}n Moln\'{a}r}
\address{E\"otv\"os Lor\'and University, Department of Logic, M\'{u}zeum krt. 4, 1088, Budapest, Hungary}
\email{mozaag@gmail.com}

\date{}

\makeatletter
\@namedef{subjclassname@2020}{\textup{2020} Mathematics Subject Classification}
\makeatother
\begin{document}
\begin{abstract}
Fix $m,n\in\omega\backslash\{0,1\}$. In set theory without the Axiom of Choice ($\mathsf{AC}$), we study the deductive strength of a generalized version of the Rainbow Ramsey theorem (``If $X$ is an infinite set and $\chi: [X]^m \rightarrow C$ is an $n$-bounded coloring for some infinite set $C$, then there is an infinite $Y \subseteq X$ which is polychromatic for
$\chi$''), we abbreviate by $\mathsf{RRT}^{m}_{n}$, and the Canonical Ramsey Theorem for pairs introduced by Erd\H{o}s and Rado, concerning their interrelation with
several weak choice forms. In this direction, we extend some work by Justin Palumbo from 2013. Moreover, we partially answer two open questions concerning the relation of Kurepa's Theorem (``Every poset such that all of its antichains are finite and all of its chains are countable is countable'') with weak choice forms.   
\end{abstract}
\maketitle

\section{Introduction}
In 1950, Erd\H{o}s and Rado extended Ramsey's theory of monochromatic graphs to allow infinitely many colors \cite{ER1950} and established the proposition   
{\em ``Given any coloring $f : [\omega]^{2} \rightarrow C$, there exists an infinite set $Y\subseteq \omega$ with an ordering $R$ such that either}
\begin{enumerate}[(a)]
{\em 
  \item $f(\{a_{1}, b_{1}\}) = f(\{a_{2}, b_{2}\})$ for all $a_{1}, b_{1}, a_{2}, b_{2} \in Y$, or
  \item $f(\{a_{1}, b_{1}\}) = f(\{a_{2}, b_{2}\})$ iff $a_{1} = a_{2}$ for all $a_{1}, b_{1}, a_{2}, b_{2} \in Y$ with $a_{1}Rb_{1}$ and $a_{2}R 
  b_{2}$, or
  \item $f(\{a_{1}, b_{1}\}) = f(\{a_{2}, b_{2}\})$ iff $b_{1} = b_{2}$ for all $a_{1}, b_{1}, a_{2}, b_{2} \in Y$ with $a_{1}Rb_{1}$ and $a_{2}R b_{2}$, or
  \item $f(\{a_{1}, b_{1}\}) = f(\{a_{2}, b_{2}\})$ iff $(a_{1} = a_{2}$ and $b_{1} = b_{2})$ for all $a_{1}, b_{1}, a_{2}, b_{2} \in Y$''
}
\end{enumerate}
in $\mathsf{ZFC}$.
We consider a natural generalization
of the above statement by replacing $\omega$ with an arbitrary infinite set $X$, and we abbreviate the generalized statement by $\mathsf{CRT}$.
In 2013, Palumbo \cite[Theorem 2.2]{Pal2013} introduced the following generalized version of the Rainbow Ramsey theorem, which we abbreviate by $\mathsf{RRT}^{2}_{2}$:
{\em ``If $X$ is an infinite set and $\chi: [X]^2 \rightarrow C$ is a $2$-bounded coloring, then there is an infinite $Y \subseteq X$ which is polychromatic for
$\chi$''}.  In \cite[Theorem 2.3]{Pal2013}, it was shown that $\mathsf{RRT}^{2}_{2}$ cannot
be proved without using some form of choice. In particular, Palumbo proved that $\mathsf{RRT}^{2}_{2}$ fails in the second Fraenkel model, and holds in the basic Cohen model (labeled as Model $\mathcal{M}_{1}$ in \cite{HR1998}) where the infinite Ramsey’s theorem for pairs ($\mathsf{RT}$) fails (see \cite[Theorems 2.3, 2.4]{Pal2013} and \cite[Theorem 1]{Bla1977}). However, Galvin's trick in \cite[Page 951]{Pal2013}, established that $\mathsf{RRT}^{2}_{2}$ follows from $\mathsf{AC}_{2}$ + $\mathsf{RT}$ in $\mathsf{ZF}$ i.e., Zermelo--Fraenkel set theory without $\mathsf{AC}$ (see Fact \ref{Fact 3.1}(1); complete definitions of the choice forms will be given in Section 2).
It is well-known that $\mathsf{CRT}$ implies $\mathsf{RRT}^{2}_{n}$ as well as $\mathsf{RT}$ in $\mathsf{ZFC}$.  Kleinberg \cite{Kle1969}, Blass \cite{Bla1977}, Forster--Truss \cite{FT2007}, and Tachtsis \cite{Tac2016} investigated the strength of $\mathsf{RT}$ in the hierarchy of weak choice forms (see Fact \ref{Fact 3.1}). In the current paper, we determine the placement of $\mathsf{RRT}^{m}_{n}$ and $\mathsf{CRT}$ in the hierarchy of weak choice forms. 

\subsection{Results.} 
Fix any $m,n\in\omega\backslash\{0,1\}$. The first author proves the following:
\begin{enumerate}
    \item $\mathsf{RRT}^{2}_{3}$ implies $\mathsf{AC}_{3}^{-}$ (``Every infinite family $\mathcal{A}$ of $3$-element sets has an infinite subfamily $\mathcal{B}$ with a choice function'') and $\mathsf{RRT}^{m}_{n}$ implies $\mathsf{AC}_{2}^{-}$ in $\mathsf{ZF}$ (Proposition \ref{Proposition 3.3}).

    \item $\mathsf{DC}$ (the principle of Dependent Choices) is equivalent to the statement ``If $G=(V_{G}, E_{G})$ is an infinite graph, then for all coloring $f:[V_{G}]^{2}\rightarrow \{0,1\}$ if all 
    $0$-monochromatic sets are finite, then there is a maximal $0$-monochromatic set'' in $\mathsf{ZF}$ (Theorem \ref{Theorem 3.5}).
     
    \item $\mathsf{RRT}^{m}_{n}$ is independent of $\mathsf{HT}$(Hindman's theorem), Kurepa's theorem stated in the abstract, $\mathsf{DT}$ (Dilworth’s theorem for infinite posets with finite
    width), ``there are no amorphous sets'', and many other combinatorial principles in $\mathsf{ZFA}$ i.e., Zermelo-Fraenkel set theory with the Axiom of Extensionality weakened to allow the existence
    of atoms (Theorem \ref{Theorem 4.1}, Remark \ref{Remark 7.3}). For recent research on $\mathsf{DT}$, 
    Kurepa's theorem, $\mathsf{HT}$, and weak choice forms the reader is referred to Tachtsis \cite{Tac2022, Tac2019,Tac2022a}, Banerjee \cite{Ban2023}, and Fern\'{a}ndez-Bret\'{o}n \cite{Fer2023}. 
    

    \item $\mathsf{DF=F}$ (Any Dedekind-finite set is finite) is strictly stronger than $\mathsf{CRT}$ in $\mathsf{ZFA}$ (Theorem \ref{Theorem 4.9}).
    \item $\Delta SL$ + $\mathsf{RRT}^{m}_{n}$ does not imply $\mathsf{RT}$ in $\mathsf{ZF}$, where $\Delta SL$ is the Delta System lemma restricted to uncountable sets. In particular, $\Delta SL$ + $\mathsf{RRT}^{m}_{n}$ holds in the basic Cohen model $\mathcal{M}_{1}$ (Theorem \ref{Theorem 5.8}(1)).  
    
    \item $\mathsf{DF=F}$ is strictly stronger than $\mathsf{RRT}^{m}_{n}$ in $\mathsf{ZF}$ (Theorem \ref{Theorem 5.8}(2)).
    
    \item $\mathsf{CRT}$ is strictly stronger than $\mathsf{RRT}^{2}_{n}$ in $\mathsf{ZF}$ (Theorem \ref{Theorem 5.8}(3)).
    
    \item $\mathsf{AC^{LO}}$ (``Every linearly ordered set of non-empty sets has a choice function'') does not imply Kurepa’s theorem in $\mathsf{ZFA}$ (Remark \ref{Remark 7.1}). 
    \item $\mathsf{WOAM}$ (``Every set is either well-orderable or has an amorphous subset'') implies Kurepa's theorem in $\mathsf{ZF}$ (Remark \ref{Remark 7.2}). 

    
\end{enumerate}
The results in (8) and (9) partially answer two questions from Tachtsis \cite{Tac2022} (see \cite[Questions 6.3, 6.4]{Tac2022}).\footnote{The ideas of these results are mainly motivated by two recent results due to Tachtsis (see \cite[Theorems 3,6]{Tac2024}).}
The result in Theorem \ref{Theorem 4.1}(1) is inspired by the arguments in the proof of \cite[Theorem 2]{Bla1977}, where Blass proved that $\mathsf{RT}$ holds in the basic Fraenkel model (labeled as Model $\mathcal{N}_{1}$ in \cite{HR1998}), whereas the result in Theorem \ref{Theorem 4.1}(2) as well as the result in (4) are inspired by the arguments in the proof of \cite[Theorem 2.4]{Tac2016}, where Tachtsis proved that $\mathsf{RT}$ holds in the Mostowski linearly ordered model (labeled as Model $\mathcal{N}_{3}$ in \cite{HR1998}). The result in (5) is inspired by the arguments in the proof of \cite[Theorem 7(i)]{Tac2018}, where Tachtsis showed that $\Delta SL$ holds in $\mathcal{N}_{1}$ as well as the arguments of Palumbo from \cite[proof of Lemma 2.6]{Pal2013}. 
In Proposition \ref{Proposition 3.3}, we also observe the following:

\begin{enumerate} 
    \item $\mathsf{RRT}^{2k+1}_{r}$ implies $\mathsf{AC}_{n}^{-}$ for all $r\geq {n \choose 2}^{k} n (k+1)$ in $\mathsf{ZF}$.
    \item $\mathsf{RRT}^{2k}_{r}$ implies $\mathsf{AC}_{n}^{-}$ for all $r\geq {n \choose 2}^{k-1} n^{2} {k+1 \choose 2}$ in $\mathsf{ZF}$.
\end{enumerate}

\subsection{Amorphous sets and Ramsey type theorems} Banerjee--Gopaulsingh \cite{BG2023} observed that $\mathsf{EDM}$ (``If $G=(V_{G}, E_{G})$ is a graph such that $V_{G}$ is uncountable, then for all coloring $f:[V_{G}]^{2}\rightarrow \{0,1\}$ either there is an uncountable set monochromatic in color $0$, or there is a countably infinite set monochromatic in color 1'') holds in $\mathcal{N}_{1}$, where the set of atoms is an amorphous set, as well as in $\mathcal{N}_{3}$, where there are no amorphous sets.\footnote{We note that $\mathsf{EDM}$ is the Erd\H{o}s--Dushnik--Miller theorem for uncountable sets.}
For recent research on $\mathsf{EDM}$, 
and weak choice forms the reader is referred to Tachtsis \cite{Tac2024} and Banerjee-Gopaulsingh \cite{BG2023}.     
Howard--Saveliev--Tachtsis \cite[Theorems 3.26, 3.9(1)]{HST2016} proved that $\mathsf{CS}$ (``Every partially ordered set without a maximal element has two disjoint cofinal subsets'') holds in $\mathcal{N}_{1}$, but fails in $\mathcal{N}_{3}$. However, the consistency of ``$\mathsf{EDM} + \mathsf{CS} + $ ``there are no amorphous subsets'''' is unknown to the best of our knowledge. 
In Theorem \ref{Theorem 6.5}, we study a {\em non-trivial argument} to observe the consistency of the statement 
\begin{center}
``$\mathsf{EDM} + \mathsf{CS} + $ ``there are no amorphous subsets'' + $\mathsf{RRT}^{m}_{n}$''
\end{center} 
in a variant of the finite partition model introduced by Bruce \cite{Bru2016} in 2016, say $\mathcal{V}_{fp}$.

\subsection{Diagram of results}
Fix any $m,n\in\omega\backslash\{0,1\}$ and $k\in \omega\backslash\{0\}$. In Figure 1, known results are depicted with dashed arrows, new implications or non-implications in $\mathsf{ZF}$ are mentioned with simple black arrows, and new non-implications in $\mathsf{ZFA}$ are depicted with thick dotted black arrows.

\begin{figure}[!ht]
\begin{minipage}{\textwidth}
\begin{tikzpicture}[scale=6]
\draw (2.15,0.71) node[above] {$\mathsf{AC}$};
\draw[dashed, -triangle 60] (2.15,0.8) -- (2.15,1);
\draw (2.15,1) node[above] {$\mathsf{AC}_{n}$};
\draw[dashed, -triangle 60] (2.15,1.1) -- (1.85,1.26);
\draw[dashed, -triangle 60] (2.15,1.1) -- (2.45, 1.26);
\draw (1.85,1.25) node[above] {$\mathsf{AC}_{n}^{-}$};
\draw (2.45, 1.25) node[above] {``There are no amorphous sets''};

\draw[dashed, -triangle 60] (2.22,0.75) -- (2.37,0.75);
\draw (2.45,0.7) node[above] {$\mathsf{DC}$};

\draw[dashed, -triangle 60] (2.45,0.8) -- (2.45,1.25);
\draw[dashed, -triangle 60] (2.52,0.75) -- (2.7,0.75);

\draw (2.9,0.7) node[above] {$\mathsf{DF=F}$};
\draw[-triangle 60] (3.05,0.75) -- (3.2,0.75);
\draw[-triangle 60] (3.2,0.72) -- (3.05,0.72);
\draw (3.16,0.74) -- (3.12,0.7);
\draw[-triangle 60] (3.05,0.75) -- (3.2, 0.9);
\draw[ultra thick, dotted, -triangle 60] (3.2, 0.86)--(3.09,0.76);
\draw  (3.14,0.83)--(3.18,0.8);
\draw[-triangle 60] (3.05,0.75) -- (3.2, 1.07);
\draw[-triangle 60] (3.05,0.75) -- (3.2, 1.25);
\draw[dashed, -triangle 60] (2.9,0.7) -- (2.9,0.56);
\draw (2.87,0.45) node[above] {$\mathsf{RT}^{m}_{n}$};
\draw[dashed, -triangle 60] (2.9,0.3) -- (2.9,0.45);
\draw (2,0.45) node[above] {$\Delta S L$+$\mathsf{RRT}^{m}_{n}$};
\draw[-triangle 60] (2.2,0.5) -- (2.32, 0.5);
\draw (2.26,0.52) -- (2.22,0.47);
\draw (2.4,0.46) node[above] {$\mathsf{RT}$};

\draw (2.9,0.2) node[above] {``$\mathsf{AC}_{\leq n}$+$\mathsf{RT}^{m}_{n}$''};
\draw[dashed, -triangle 60] (3,0.3) -- (3.3,0.7);
\draw[-triangle 60] (2.78,0.5) -- (2.48,0.5);

\draw (3.33,0.68) node[above] {$\mathsf{RRT}^{m}_{n}$};
\draw[-triangle 60] (3.45,0.75) -- (3.88,0.75);
\draw (4,0.67) node[above] {$\mathsf{AC^{-}_{2}}$};

\draw (3.28,0.85) node[above] {$\mathsf{CRT}$};
\draw[-triangle 60] (3.35,0.9) -- (3.5,0.9);
\draw[-triangle 60] (3.5,0.87) -- (3.35,0.87);
\draw (3.44,0.89) -- (3.4,0.85);
\draw (3.75,0.83) node[above] {$(\forall n\geq 2)\mathsf{RRT}^{2}_{n}$};
\draw[-triangle 60] (3.97,0.9) -- (4.07,0.9);
\draw (4.17,0.83) node[above] {$\mathsf{RRT}^{2}_{3}$};
\draw[-triangle 60] (4.27,0.9) -- (4.37,0.9);
\draw (4.45,0.83) node[above] {$\mathsf{AC^{-}_{3}}$};

\draw (3.42,0.97) node[above] {$\mathsf{RRT}^{2k+1}_{{n \choose 2}^{k} n (k+1)}$};
\draw[-triangle 60] (3.62,1.05) -- (3.88,1.2);
\draw (4,1.2) node[above] {$\mathsf{AC}^{-}_{n}$};

\draw (3.45,1.17) node[above] {$\mathsf{RRT}^{2k}_{{n \choose 2}^{k-1} n^{2} {k+1 \choose 2}}$};
\draw[-triangle 60] (3.69,1.25) -- (3.9,1.25);

\draw[ultra thick, dotted, -triangle 60] (4.2,0.2) -- (3.4,0.7);
\draw[ultra thick, dotted, -triangle 60] (3.35,0.67) -- (4.1,0.2);
\draw (3.1,0.05) node[above] {$X$ if $X\in\{\mathsf{DT},\mathsf{CAC^{\aleph_{0}}},\mathsf{CAC_{1}^{\aleph_{0}}},\mathsf{WOAM}, $ ``Antichain Principle'', $\mathsf{CS}$, $\mathsf{MC},$ $\nexists$ amorphous sets$,\mathsf{HT}\}$};
\draw (3.8,0.42) -- (3.76,0.37);
\draw (3.8,0.48) -- (3.76,0.43);

\draw[thick] (1.66,0.05) rectangle (4.52,0.18);

\end{tikzpicture}
\end{minipage}
\caption{\em Implications/non-implications between the principles.}
\end{figure}
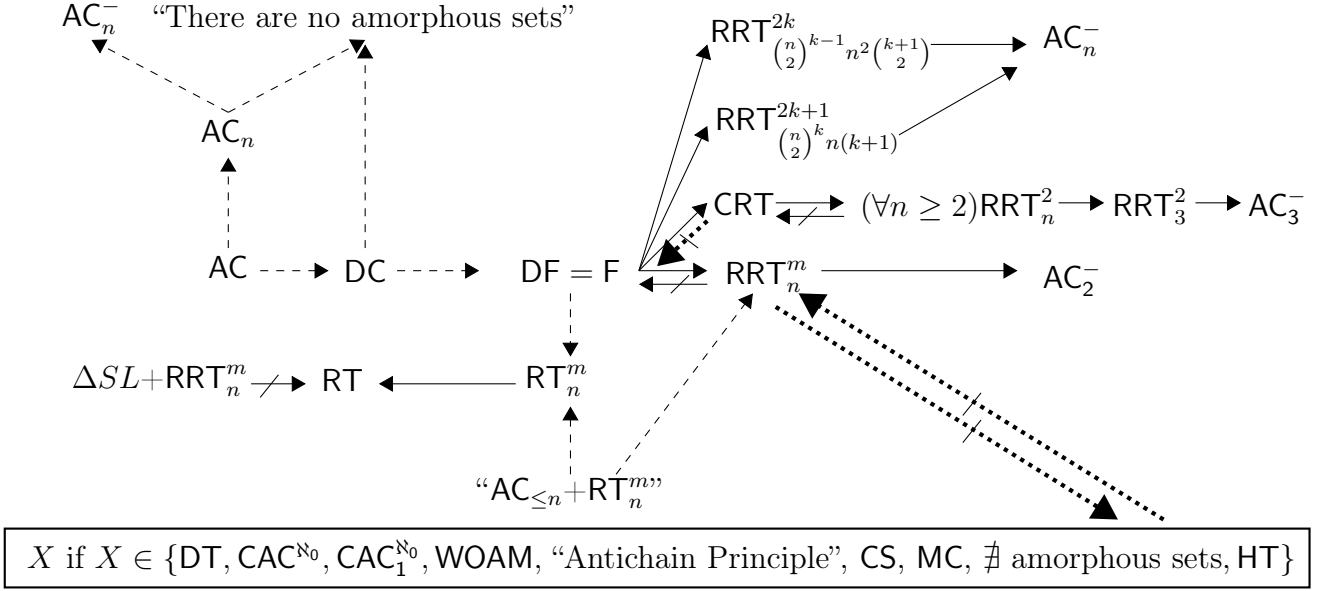

\section{Basics}
\begin{defn}
As usual, $\omega$ denotes the set of natural numbers. For any set $X$ and $n\in\omega$, the set of all $n$-element subsets of $X$ is denoted by $[X]^{n}$ and the set of all finite subsets of $X$ is denoted by $[X]^{<\omega}$. Let $(P, \leq)$ be a partially ordered set, or ‘poset’ in short. 
A subset $D \subseteq P$ is a {\em chain} if $(D, \leq\restriction D)$ is linearly ordered. A subset $A\subseteq P$ is an {\em antichain}
if no two elements of $A$ are comparable under $\leq$. The size of the largest antichain of $(P, \leq)$ is known as its {\em width}.  A subset $C \subseteq P$ is {\em cofinal} in $P$ if for every $x \in P$ there is an element $c \in C$ such that $x \leq c$. A set $X$ is called {\em Dedekind-finite} if $\aleph_{0} \not
\leq \vert X\vert$ i.e., if there is no one-to-one function $f : \omega \rightarrow X$. Otherwise, $X$ is called {\em Dedekind-infinite}. An infinite set $X$ is {\em amorphous} if $X$ cannot be written as a
disjoint union of two infinite subsets.  
A set $X$ is called {\em uncountable} if $\vert X\vert \not\leq \aleph_{0}$. A family $\mathcal{A}$ of sets is called a {\em $\Delta$-system} if there is a set $r$ such that for any two
distinct elements $x, y$ of $\mathcal{A}$, $x \cap y = r$. The set $r$ is called the {\em root} of the $\Delta$-system.
The coloring $\chi: [X]^{m} \rightarrow C$ is {\em $n$-bounded} if $\vert \chi^{-1}[c]\vert\leq n$ for each $c\in C$.
If $\chi: [X]^{m} \rightarrow C$ is a coloring, then $Y\subseteq X$ is {\em monochromatic for $\chi$} if for all $y_{1},y_{2}\in [Y]^{m}$, $\chi(y_{1})=\chi(y_{2})$, i.e., there is a single color which all elements of $[Y]^{n}$ receives, and $Z\subseteq X$ is {\em polychromatic for $\chi$} if for all $z_{1},z_{2}\in [Z]^{m}$, $\chi(z_{1})\neq\chi(z_{2})$, i.e., each member of $[Z]^{n}$ receives a different color. A graph $G=(V_{G}, E_{G})$ consists of a set $V_{G}$ of vertices and a set $E_{G}\subseteq [V_{G}]^{2}$ of edges. An {\em independent set} of $G$ is a set of vertices of $G$, no two of which are adjacent vertices. A {\em maximal independent set} is an independent set that is not a subset of any other independent set.

\end{defn}
\begin{defn}{(A list of combinatorial statements).} Fix $n, m\in \omega\backslash\{0,1\}$.
\begin{enumerate}
\item The {\em rainbow Ramsey theorem}, $\mathsf{RRT}^{m}_{n}$: If $X$ is an infinite set and $\chi: [X]^m \rightarrow C$ is an $n$-bounded coloring for some infinite set $C$, then there is an infinite set $Y \subseteq X$ which is polychromatic for
$\chi$.

\item {\em Ramsey's Theorem}, $\mathsf{RT}^{m}_{n}$: Let $X$ be an infinite set and let $\chi : [X]^{m} \rightarrow n$ be a coloring. Then there is an infinite $Y \subseteq X$ which is monochromatic for $\chi$.  

\item The {\em Canonical Ramsey Theorem for pairs}, $\mathsf{CRT}$: Given any coloring $f : [X]^{2} \rightarrow C$, there exists an infinite set $Y\subseteq X$ with a linear ordering $R$ such that either
\begin{enumerate}[$(a)$]
  \item $f(\{a_{1}, b_{1}\}) = f(\{a_{2}, b_{2}\})$ for all $a_{1}, b_{1}, a_{2}, b_{2} \in Y$, or
  \item $f(\{a_{1}, b_{1}\}) = f(\{a_{2}, b_{2}\})$ iff $a_{1} = a_{2}$ for all $a_{1}, b_{1}, a_{2}, b_{2} \in Y$ with $a_{1}Rb_{1}$ and $a_{2}R 
  b_{2}$, or
  \item $f(\{a_{1}, b_{1}\}) = f(\{a_{2}, b_{2}\})$ iff $b_{1} = b_{2}$ for all $a_{1}, b_{1}, a_{2}, b_{2} \in Y$ with $a_{1}Rb_{1}$ and $a_{2}R b_{2}$, or
  \item $f(\{a_{1}, b_{1}\}) = f(\{a_{2}, b_{2}\})$ iff $(a_{1} = a_{2}$ and $b_{1} = b_{2})$ for all $a_{1}, b_{1}, a_{2}, b_{2} \in Y$.
\end{enumerate}

\item The {\em Erd\H{o}s--Dushnik--Miller theorem}, $\mathsf{EDM}$: If $G=(V_{G}, E_{G})$ is a graph where $V_{G}$ is uncountable, then for all coloring $f:[V_{G}]^{2}\rightarrow \{0,1\}$ either there is an uncountable set monochromatic in color $0$, or there is a countably infinite set monochromatic in color 1.

\item The {\em Chain-Antichain principle}, $\mathsf{CAC}$: Every infinite poset has either an infinite chain or an infinite antichain.

\item {\em Kurepa's theorem}, $\mathsf{CAC_{1}^{\aleph_{0}}}$:  Every poset such that all of its antichains are finite and all of its chains are countable is countable.

\item A variant of Kurepa's theorem, $\mathsf{CAC^{\aleph_{0}}}$:  Every poset such that all of its chains are finite and all of its antichains are countable is countable.

\item {\em Hindman’s theorem}, $\mathsf{HT}$: For every infinite set $X$ and for every coloring $c : [X]^{<\omega} \rightarrow 2$, there exists an infinite, pairwise disjoint family $Y \subseteq [X]^{<\omega}$ such that the set
\begin{center}
    $FU(Y ) = \{\bigcup_{y\in F}y : F \in [Y]^{<\omega} \backslash \{\emptyset\}\}$
\end{center}
is monochromatic for $c$. 

\item  {\em $\Delta$-system Lemma}, $\Delta SL$: For every uncountable family $\mathcal{A}$ of finite sets, there is an uncountable subfamily $\mathcal{B}$ of $\mathcal{A}$ which forms a $\Delta$-system. 

\item The {\em Antichain Principle, $\mathsf{A}$}: Every poset has a maximal antichain.

\item {\em Dilworth’s Theorem}, $\mathsf{DT}$: If $(P,\leq)$ is a poset of width $k$ for some $k\in \omega$, then $P$ can be partitioned into $k$ chains.

\item {\em $\mathsf{CS}$}: Every poset without a maximal element has two disjoint cofinal subsets.

\end{enumerate} 
\end{defn}

\begin{defn}{(A list of choice forms).}

\begin{enumerate}
\item The {\em Axiom of Choice}, $\mathsf{AC}$ \cite[Form 1]{HR1998}: Any family of non-empty sets has a choice function.

\item The {\em Principle of Dependent Choices}, $\mathsf{DC}$ \cite[Form 43]{HR1998}: Let $S$ be a non-empty set and let $R$ be an {\em entire} relation on $S$ (i.e., a binary relation $R$ on $S$ such that $(\forall x \in S)(\exists y \in S)(xRy)$). Then, there exists a sequence $(x_{n})_{n\in\omega}$ of elements of $S$ such
that $x_{n}Rx_{n+1}$ for all $n \in \omega$. 

\item The {\em Principle of Dependent
Choices for $\kappa$}, $\mathsf{DC_{\kappa}}$, 
where $\alpha$ is the ordinal such that $\kappa=\aleph_{\alpha}$
\cite[Form 87($\alpha$)]{HR1998}: Let $S$ be a non-empty set and let $R$ be a binary relation such that for every $\beta<\kappa$ and every $\beta$-sequence $s =(s_{\epsilon})_{\epsilon<\beta}$ of elements of $S$ there exists $y \in S$ such that $s R y$. Then there is a function $f : \kappa \rightarrow S$ such that for every $\beta < \kappa$, $(f\restriction \beta) R f(\beta)$. We note that $\mathsf{DC_{\aleph_{0}}}$ is a reformulation of $\mathsf{DC}$, and $\mathsf{DC_{\aleph_{1}}}$ is strictly stronger than $\mathsf{DC}$ in $\mathsf{ZF}$.

\item $\mathsf{DF = F}$ \cite[Form 9]{HR1998}: Every Dedekind-finite set is finite.
We note that $\mathsf{DC}$ is strictly stronger than $\mathsf{DF=F}$ in $\mathsf{ZF}$.

\item $\mathsf{AC}_{n}^{-}$ for each $n\in\omega\backslash \{0,1\}$ \cite[Form 342($n$)]{HR1998}: Every infinite family $\mathcal{A}$ of $n$-element sets has a {\em partial choice function}, i.e., $\mathcal{A}$ has an infinite subfamily $\mathcal{B}$ with a choice function. 

\item  $\mathsf{AC}_{
\leq n}$ for each $n\in\omega\backslash \{0,1\}$: Every infinite family of non-empty sets, each with at most $n$
elements, has a choice function.

\item The {\em Ordering Principle}, $\mathsf{OP}$ \cite[Form 30]{HR1998}: Every set can be linearly ordered.

\item The {\em Axiom of Choice for finite Sets}, $\mathsf{AC}_{fin}$ \cite[Form 62]{HR1998}:  Every family of non-empty, finite sets has a choice function.

\item The {\em Axiom of Multiple Choice}, $\mathsf{MC}$ (\cite[Form 67]{HR1998}): Every family $\mathcal{A}$ of non-empty sets has a {\em multiple choice function}, i.e., there is a function $f$ with domain $\mathcal{A}$ such that for every $A \in \mathcal{A}$, $\emptyset\neq f(A)\in [A]^{<\omega}$.

\item The {\em Countable Union Theorem}, $\mathsf{CUT}$ \cite[Form 31]{HR1998}: The union of a countable family of countable sets is countable.

\item  $\mathsf{WUT}$ \cite[Form 231]{HR1998}: The union of a well-orderable family of well-orderable
sets is well-orderable.
\item $\mathsf{WOAM}$ \cite[Form 133]{HR1998}: Every set is either well-orderable or has an amorphous
subset.
\item $\mathsf{PC}$: Every uncountable family of countable sets has an uncountable subfamily
with a choice function.
\item $\mathsf{AC^{LO}}$ \cite[Form 202]{HR1998}: Every linearly ordered set of non-empty sets has a 
choice function.

\item $\mathsf{AC^{WO}_{fin}}$ \cite[Form 122]{HR1998}: Any well-ordered set of non-empty finite sets has a choice function.
\end{enumerate}
\end{defn}
\subsection{Permutation models} We provide a brief account of the construction of Fraenkel-Mostowski permutation models of $\mathsf{ZFA}$ from \cite[Chapter 4]{Jec1973}. Let $M$ be a model of $\mathsf{ZFA+AC}$ where $A$ is a set of atoms. Let $\mathcal{G}$ be a group of permutations of $A$ and $\mathcal{F}$ be a normal filter of subgroups of $\mathcal{G}$. 
For a set $x\in M$, we put $sym_{\mathcal {G}}(x) =\{g\in \mathcal {G} \mid g(x) = x\}$, fix$_{\mathcal{G}}(x)=\{\phi \in \mathcal{G} : \forall y \in x (\phi(y) = y)\}$, and  $\mathsf{TC}(x)$ is the transitive closure of $x$ in $M$. 
\begin{enumerate}
\item The permutation model $\mathcal{N}_{\mathcal{F}}$ with respect to $M$, $\mathcal{G}$  and $\mathcal{F}$ is defined by the equality:
\begin{center}

    $\mathcal{N}_{\mathcal{F}} = \{x \in M : (\forall t \in \mathsf{TC}(\{x\}))(sym_{\mathcal{G}}(t) \in \mathcal{F})\}$.
\end{center}
\item The permutation model $\mathcal{N}_{\mathcal{I}}$ with respect to $M$, $\mathcal{G}$ and a normal ideal $\mathcal{I}\subseteq\mathcal{P}(A)$ is defined by the equality:
\begin{center}

    $\mathcal{N}_{\mathcal{I}} = \{x \in M : (\forall t \in \mathsf{TC}(\{x\}))(\exists E \in \mathcal{I})($fix$_{\mathcal{G}}(E) \subseteq sym_{\mathcal{G}}(t))\}$.
\end{center}
\end{enumerate}

We recall that $\mathcal{N}_{\mathcal{F}}$ and $\mathcal{N}_{\mathcal{I}}$ are models of $\mathsf{ZFA}$ (cf. \cite[Theorem 4.1, page 46]{Jec1973}). We say $E\in \mathcal{I}$ is a {\em support} of a set $\sigma\in \mathcal{N}_{\mathcal{I}}$ if fix$_{\mathcal{G}}(E)\subseteq sym_{\mathcal{G}} (\sigma$). In this paper, we follow the labeling of the models from \cite{HR1998}. $\mathcal{N}_{1}$ is the basic Fraenkel model, $\mathcal{N}_{2}$ is the second Fraenkel model, $\mathcal{N}_{3}$ is the Mostowski linearly ordered model, and $\mathcal{N}_{26}$ is Brunner/Pincus’s Model (cf. \cite{HR1998}).
\begin{lem}\label{Lemma 2.4}
{\em An element $x$ of $\mathcal{N}_{\mathcal{I}}$ is well-orderable in $\mathcal{N}_{\mathcal{I}}$ if and only if {\em fix}$_{\mathcal{G}}(x)\in \mathcal{F}_{\mathcal{I}}$ where $\mathcal{F}_{\mathcal{I}}$ is the normal filter generated by the filter base $\{${\em fix}$_{\mathcal{G}}(E): E\in\mathcal{I}\}$ {\em (cf. \cite[Equation (4.2), page 47]{Jec1973})}. Thus, an element $x$ of $\mathcal{N}_{\mathcal{I}}$ with support $E$ is well-orderable in $\mathcal{N}_{\mathcal{I}}$ if {\em fix}$_{\mathcal{G}}(E) \subseteq$ {\em fix}$_{\mathcal{G}}(x)$.
}
\end{lem}

\begin{lem}\label{Lemma 2.5}
{(see \cite[Proposition 3.5]{BG2023})} {\em Let $A$ be a set of atoms. Let $\mathcal{G}$ be the group of permutations of $A$ such that each $\eta\in \mathcal{G}$ moves only finitely many atoms. Let $\mathcal{N}$ be the permutation model determined by $A$, $\mathcal{G}$, and a normal filter $\mathcal{F}$ of subgroups of $\mathcal{G}$. Then $\mathsf{CS}$ and the Antichain Principle $\mathsf{A}$ hold in $\mathcal{N}$.} 
\end{lem}

\section{Positive results and known results}
\begin{fact}\label{Fact 3.1}
{\em 
\begin{enumerate} 
\item $\mathsf{AC}_{\leq n}+\mathsf{RT}^{m}_{n}$ implies $\mathsf{RRT}^{m}_{n}$ in $\mathsf{ZF}$. The proof is due to Galvin (see \cite[Page 951]{Pal2013}). In particular, let $\chi:[X]^m \rightarrow C$ be an $n$-bounded coloring. By $\mathsf{AC}_{\leq n}$, for each $c \in C$, we can fix an enumeration of $\chi^{-1}[c]$ and form the coloring $\chi':[X]^m \rightarrow n$ by letting $\chi'(a) = i$  where $a$ is the $i^{th}$ element in the enumeration of its color class. By $\mathsf{RT}^{m}_{n}$, there exists an infinite subset $Y$ of $X$  which is monochromatic for $\chi'$. Clearly, $Y \subseteq X$ is polychromatic for $\chi$. 

\item $\mathsf{RT}^{m}_{n}$ holds for well-orderable sets in $\mathsf{ZF}$.

\item $\mathsf{RRT}^{m}_{n}$ holds for well-orderable sets in $\mathsf{ZF}$. This follows by (2), and the above-mentioned trick due to Galvin from (1).
Moreover, if $\mathsf{RRT}^{m}_{n}$ holds for an infinite set $Y$, then $\mathsf{RRT}^{m}_{n}$ holds for any set $X \supseteq Y$ in $\mathsf{ZF}$.

\item $\mathsf{RRT}^{m}_{n}$ implies $\mathsf{RRT}^{m}_{k}$  in $\mathsf{ZF}$ if $k<n$ since  any $k$-bounded coloring is also an $n$-bounded coloring.

\item (Palumbo; \cite[Theorem 2.4 and Proposition 2.7]{Pal2013}) $\mathsf{RRT}^{2}_{2}$ holds in the basic Cohen model $\mathcal{M}_{1}$. Moreover, $\mathsf{AC}_{\leq n}+\mathsf{RRT}^{2}_{2}$ implies $\mathsf{RRT}^{2}_{n}$ in $\mathsf{ZF}$.

\item (Blass; \cite[Theorems 1 and 2]{Bla1977}) $\mathsf{RT}^{2}_{2}$ fails in $\mathcal{M}_{1}$ and holds in $\mathcal{N}_{1}$.

\item (Tachtsis; \cite[Theorem 2.4]{Tac2016}) $\mathsf{RT}^{2}_{2}$ holds in $\mathcal{N}_{3}$.

\item (Forster–Truss; \cite[Lemma 2.2, Theorem 2.3]{FT2007}) In $\mathsf{ZF}$, for each fixed $m\in\omega\backslash\{0,1\}$, the statements $\mathsf{RT}^{m}_{n}$ are equivalent for all $n\in\omega\backslash\{0,1\}$. Moreover, if $n_{1} \geq n_{2} \geq 1$ and $k_{1}, k_{2} > 1$, then $\mathsf{RT}^{n_{1}}_{k_{1}}$ implies $\mathsf{RT}^{n_{2}}_{k_{2}}$.

\item (Tachtsis; \cite[Lemma 4]{Tac2018}, \cite[Lemma 3.12]{Tac2021}) In $\mathcal{N}_{1}$ and in $\mathcal{N}_{3}$, any non-well-orderable family of non-empty sets has a non-well-orderable subfamily with a choice function.
\item (Howard–Solski; \cite[Corollary 2.5]{HS1993})
$\Delta SL$ is equivalent to $\mathsf{CUT} + \mathsf{PC}$ in $\mathsf{ZF}$.

\item If $\mathsf{CRT}$ holds for an infinite set $Y$, then $\mathsf{CRT}$ holds for any set $X \supseteq Y$ in $\mathsf{ZF}$.
\item (Banerjee--Gopaulsingh; \cite[Proposition 3.3 (4)]{BG2023}) $\mathsf{EDM}$ restricted to graphs based on a well-ordered set of vertices holds in any permutation model.  
\end{enumerate}
}
\end{fact}

\subsection{Positive results}

\begin{prop}\label{Proposition 3.2}{($\mathsf{ZF}$)}
{\em Fix $m,n\in\omega\backslash\{0,1\}$. The following hold:
\begin{enumerate}
    \item $\mathsf{DF=F}$ implies $\mathsf{CRT}$.
    \item $\mathsf{DF=F}$ implies $\mathsf{RRT}^{m}_{n}$.
    \item $\mathsf{CRT}$ implies $\mathsf{RRT}^{2}_{n}$ and $\mathsf{RT}^{2}_{n}$.
\end{enumerate}    
}
\end{prop}

\begin{proof}
(1). Let $X$ be an infinite set, and $\chi:[X]^{2}\rightarrow C$ be a coloring. Let $X'$ be a countably infinite subset of $X$ by $\mathsf{DF=F}$. Fix a well-ordering $<_{X'}$ of $X'$. We show that $\mathsf{CRT}$ holds for $X'$.
Consider the coloring $\chi':[X']^{4}\rightarrow 203$ such that 
if $\{x_{1},x_{2},x_{3},x_{4}\}\in [X']^{4}$ with $x_{1}<_{X'}x_{2}<_{X'}x_{3}<_{X'}x_{4}$, then $\chi'(\{x_{1},x_{2},x_{3},x_{4}\})\in \{0,1,...,202\}$ depending on the 203 different possibilities of equalities on the values of $\chi$ on $[\{x_{1},x_{2},x_{3},x_{4}\}]^{2}$ as described in Graham--Rothschild--Spencer \cite[Proof of Theorem 2 in section 5.5, page 129]{GRS1980}.
Since $\mathsf{RT}^{4}_{203}$ holds for well-orderable sets in $\mathsf{ZF}$ by Fact \ref{Fact 3.1}(2), we can obtain a countably infinite set $Y'\subseteq X'$ which is monochromatic for $\chi'$. Define $R=<_{X'}\restriction Y'$.
We can follow the arguments of  \cite[Proof of Theorem 2 in section 5.5]{GRS1980} which works without invoking any form of choice after the use of $\mathsf{RT}^{4}_{203}$, to show that $Y'$ and $R$
satisfies either of ($(a)$-$(d)$), i.e., the consequences mentioned in the definition of $\mathsf{CRT}$ (see Definition 2.2(3)). Since $X'\subseteq X$, the conclusion follows by Fact 
 \ref{Fact 3.1}(11).

(2). Let $\chi:[X]^{m}\rightarrow C$ be an $n$-bounded coloring. By $\mathsf{DF=F}$, $X$ has a countably infinite subset, say $X'$. 
By Fact \ref{Fact 3.1}(3), there exists an infinite subset $Y'$ of $X'$ which is polychromatic for $\chi\restriction [X']^{m}$ in $\mathsf{ZF}$. Thus, $Y'$ is an infinite subset of $X$ which is polychromatic for $\chi$. 

(3). This is straightforward. 
\end{proof}

\begin{prop}\label{Proposition 3.3}{($\mathsf{ZF}$)}
{\em Fix $m,n\in\omega\backslash\{0,1\}$ and $k\in \omega\backslash\{0\}$. The following hold:
\begin{enumerate}
    \item $\mathsf{RRT}^{m}_{n}$ implies $\mathsf{AC}_{2}^{-}$.
    \item $\mathsf{RRT}^{2k+1}_{r}$ implies $\mathsf{AC}_{n}^{-}$ for all $r\geq {n \choose 2}^{k} n (k+1)$.
    \item $\mathsf{RRT}^{2k}_{r}$ implies $\mathsf{AC}_{n}^{-}$ for all $r\geq {n \choose 2}^{k-1} n^{2} {k+1 \choose 2}$.
    \item $\mathsf{RRT}^{2}_{3}$ implies $\mathsf{AC}_{3}^{-}$.

\end{enumerate}
}
\end{prop}

\begin{proof}

Let $T(n,k)={n \choose 2}^{k} n (k+1)$, and $R(n,k)={n \choose 2}^{k-1} n^{2} {k+1 \choose 2}$. In view of Fact \ref{Fact 3.1}(4), it is enough to show that $\mathsf{RRT}^{m}_{2}$ implies $\mathsf{AC_{2}^{-}}$ in (1),  $\mathsf{RRT}^{2k+1}_{T(n,k)}$ implies $\mathsf{AC}_{n}^{-}$ in (2), and $\mathsf{RRT}^{2k}_{R(n,k)}$ implies $\mathsf{AC}_{n}^{-}$ in (3).

(1) Let, $\mathcal{A}=\{A_{i}: i\in I\}$ be a collection of $2$-element sets $A_{i}=\{a_{i_{1}},a_{i_{2}}\}$ without a partial choice function and let $A=\bigcup_{i\in I}A_{i}$. 

\begin{claim}\label{claim:2_bound}
{\em There exists a $2$-bounded coloring $f: [A]^{m}\rightarrow C$ for some color set $C$ such that
for all distinct $p_{1},p_{2},...,p_{m}\in I$ and all $1\leq q_{1},...,q_{m}\leq 2$ we have}
$$f(\{a_{p_{1}q_{1}},...,a_{p_{m}q_{m}}\})=f(\{a_{p_{1}(3-q_{1})},...,a_{p_{m}(3-q_{m})}\}).$$
\end{claim}
\begin{proof}
Define $K=\{a\in [A]^m: (\exists P\in [I]^m)(\forall i\in P)|a\cap A_i|= 1\}$. 
 Introduce the equivalence relation $\sim$ on $K$ as
 $$a\sim b \iff a=b \text{ or } (\exists P\in [I]^m) (\forall i\in P) a\cap b\cap A_i =\emptyset \text{ and } a\cap A_i\neq \emptyset \text{ and } b\cap A_i\neq \emptyset.$$
 Let $C=\{a/\!\sim: a\in K\}\cup ([A]^m\setminus K)$ be the color set, where $a/\!\sim$ is the equivalence class of $a\in K$.  Let $f: [A]^m \to C$ be a mapping such that
 \begin{align*}
     f(a) =\begin{cases} a & \text{ if } a\not\in K\\
     a/\!\sim & \text{ otherwise.}
     \end{cases}
 \end{align*}
 Clearly, $f$ is a $2$-bounded coloring with the required property.
\end{proof}

Let $f: [A]^{m}\rightarrow C$ be a $2$-bounded coloring as in claim \ref{claim:2_bound}. By $\mathsf{RRT}^{m}_{2}$, there exists an infinite subset $B$ of $A$ which is polychromatic for $f$.
Let $M_{m}=\{i\in I: \vert A_{i} \cap B\vert =m\}$ for $1\leq m\leq 2$. Clearly, $M_{1}$ is finite. Otherwise, 
$\{A_{i} \cap B : i\in M_{1}\}$ will determine a partial choice function for $\mathcal{A}$. Thus $M_2$ has to be infinite, hence there exists $P\in [M_2]^{m}$ such that $\vert A_{p}\cap B\vert\geq 2$ for each $p\in P$, which contradicts the fact that $B$ is polychromatic for $f$.

(2) Let $\mathcal{A}=\{A_{i}: i\in I\}$ be a collection of $n$-element sets $A_{i}=\{a_{i_{1}},...,a_{i_{n}}\}$ without a partial choice function and let $A=\bigcup_{i\in I}A_{i}$. 
For all $P\in [I]^{k+1}$, let 
\begin{align}
    \begin{split}
X_{P}=\{\{B_{p}:p\in P\}: &(\forall p\in P) (B_{p}\in [A_{p}]^{2} \lor B_{p}\in [A_{p}]^1)\\ 
& \land (\exists S\in [P]^{k})((\forall p\in S) (B_{p}\in [A_{p}]^{2}) \land (\forall p\in P\backslash S) (B_{p}\in [A_{p}]^1))\}\text{\footnotemark}
\end{split} 
\end{align}\footnotetext{i.e. if $P=\{p_{1},...,p_{k+1}\}$, then $X_P$ is the set of all possible $k+1$ element sets $\{B_{p_1},\dots, B_{p_{k+1}}\}$ such that exactly one element (say $B_{p_{i}}$) is in $[A_{p_{i}}]^1$ and  all other elements  $B_{p_{j}}$ is in $ [A_{p_{j}}]^2$.}Then for any $P\in [I]^{k+1}$, we have 
$\vert X_{P}\vert$=$T(n,k)$.
Similarly to claim \ref{claim:2_bound} we can construct a $T(n,k)$-bounded coloring  $f: [A]^{2k+1}\rightarrow C$  such that 
for all $P\in [I]^{k+1}$, and all $\{D_{p}:p\in P\},\{E_{p}:p\in P\}\in X_{P}$, we have
$f(\bigcup_{p\in P} D_{p})=f(\bigcup_{p\in P} E_{p})$. 
By $\mathsf{RRT}^{2k+1}_{T(n,k)}$, there exists an infinite subset $B$ of $A$ which is polychromatic for $f$.
Let $M_{m}=\{i\in I: \vert A_{i} \cap B\vert =m\}$ for each $1\leq m\leq n$. 
Clearly, $M_{1}$ is finite as in (1). Since, $B=\bigcup_{1\leq m\leq n} (\cup_{i\in M_{m}} A_{i}\cap B)$ is infinite, and $M_{1}$ is finite,  we have that $\cup_{1<p\leq n} M_{p}$ is infinite since the finite union of finite sets is finite in $\mathsf{ZF}$. 
Thus there exists $P\in [I]^{k+1}$ such that $\vert A_{p}\cap B\vert\geq 2$ for each $p\in P$, which contradicts the fact that $B$ is polychromatic for $f$.

(3) 
Assume $\mathcal{A}$ and $A$ as in (2). 
For all $P\in [I]^{k+1}$, let 
\begin{equation}
\begin{split}
X_{P}=\{\{B_{p}:p\in P\}: &(\forall p\in P) (B_{p}\in [A_{p}]^{2} \lor B_{p}\in [A_{p}]^1)\\ 
& \land (\exists S\in [P]^{k-1})((\forall p\in S) (B_{p}\in [A_{p}]^{2}) \land (\forall p\in P\backslash S) (B_{p}\in [A_{p}]^1))\}.
\end{split}    
\end{equation}
Then, 
$\vert X_{P}\vert$=$R(n,k)$ for any $P\in [I]^{k+1}$.
Similarly to claim \ref{claim:2_bound} we can construct an $R(n,k)$-bounded coloring $f: [A]^{2k}\rightarrow C$ such that 
for all $P\in [I]^{k+1}$, and all $\{D_{p}:p\in P\},\{E_{p}:p\in P\}\in X_{P}$, we have
$f(\bigcup_{p\in P} D_{p})=f(\bigcup_{p\in P} E_{p})$. 
The rest follows by applying $\mathsf{RRT}^{2k}_{R(n,k)}$, and by following the arguments of (2).

(4) Assume $\mathcal{A}$ and $A$ as in (2) by replacing $n$ by $3$. Similarly to claim \ref{claim:2_bound} we can construct a $3$-bounded coloring $f: [A]^{2}\rightarrow C$ such that 
$f(\{a_{pq_{1}},a_{pq_{2}}\})=f(\{a_{pr_{1}},a_{pr_{2}}\})$     
for all $p\in I$ and all $1\leq q_{1},q_{2}, r_{1},r_{2}\leq 2$ where $q_{1}\neq q_{2}$ and $r_{1}\neq r_{2}$. By $\mathsf{RRT}^{2}_{3}$, there exists an infinite subset $B$ of $A$ which is polychromatic for $f$. Assume $M_{m}$ for each $1\leq m\leq 3$ as in (1). Clearly, $M_{1}\cup M_{2}$ is finite. Otherwise, 
$\{A_{i} \cap B : i\in M_{1}\}\cup\{A_{i}\backslash (A_{i} \cap B) : i\in M_{2}\}$
determines a partial choice function for $\mathcal{A}$. So, there exists $p\in I$ such that $\vert A_{p}\cap B\vert= 3$, 
which contradicts the fact that $B$ is polychromatic for $f$. 
\end{proof}

\begin{defn}\label{Definition 3.4}
We denote the length of a finite sequence $x$ as $\vert x\vert$ and its elements as
$x=\{x^{i}\}^{\vert x\vert}_{i=1}$. 
\end{defn}

\begin{thm}\label{Theorem 3.5}
{($\mathsf{ZF}$)}
{\em The following are equivalent:
\begin{enumerate}
\item $\mathsf{DC}$
\item If $G=(V_{G}, E_{G})$ is an infinite graph, then for all coloring $f:[V_{G}]^{2}\rightarrow \{0,1\}$ if all 
$0$-monochromatic sets are finite, then there is a maximal $0$-monochromatic set.
\item If all independent sets of an infinite graph $G$ are finite, then there exists a maximal independent set of $G$.
\end{enumerate}
}
\end{thm}

\begin{proof}(2$\Leftrightarrow$3) is straightforward.

$(1\Rightarrow 3)$ Assume $\mathsf{DC}$. Suppose there are no maximal independent sets in $G$. We will show that there exists an infinite independent set in $G$. Let
$S^{ind}_{fin}$ be a family of all
finite independent sets of $G$. Clearly, $S^{ind}_{fin}\neq \emptyset$ as $\{v\}\in S^{ind}_{fin}$ if $v\in V_{G}$. 
Define a relation
$R$ on $S^{ind}_{fin}$ as follows:
    $xRy$ $\Leftrightarrow \vert y \vert = \vert x\vert + 1$ and $x\subseteq y$. 
The relation $R$ is entire. Fix
$x\in S^{ind}_{fin}$. Since $x$ is not maximal, there exists $v\in V_{G}\backslash x$ such that $x\cup \{v\}\in S^{ind}_{fin}$. Thus, $x R (x\cup\{v\})$. By $\mathsf{DC}$, there exists a sequence $\{x_{i}\}_{i\in\omega}$ such that
$x_{i} R x_{i+1}$ for all $i\in \omega$. Then $\bigcup_{i\in \omega} x_{i}$ is an infinite independent set of $G$.

$(3\Rightarrow 1)$ Let $S$ be a non-empty set, and $R$ be an entire relation on $S$. Consider the following graph $G=(V_{G}, E_{G})$:
\begin{flushleft}
$V_{G}:= S_{fin}$ where $S_{fin}$ is the family of all finite sequences in $S$,

$E_{G}:= \{\{x,y\}: x,y\in V_{G}, \neg (x\overline{R}y), \neg(y\overline{R}x)\}$ where $\overline{R}$ is a binary relation on $S_{fin}$ such that

$x\overline{R} y\Leftrightarrow \vert x\vert < \vert y\vert, x^{i}=y^{i}$ for all $i\in \{1,2,...,\vert x\vert\},$ and $y^{i}Ry^{i+1}$ for all $i\in \{\vert x\vert,...,\vert y\vert-1\}$.
\end{flushleft}

\begin{claim}\label{claim 3.6}
{\em All maximal independent sets in $G$ are infinite.}
\end{claim}

\begin{proof}
Assume $Y=\{y_{1},...,y_{k}\}$ is a finite maximal independent set in $G$. Let $y_{n}\in Y$ be the sequence with the greatest length. Since $R$ is entire, there exists an $x^{1}\in S$ such that $y_{n}^{\vert y_{n}\vert}Rx^{1}$. If $x=(y_{n}^{1},..., y_{n}^{\vert y_{n}\vert}, x^{1})\in S_{fin}$, then $y_{n}\overline{R}x$. Thus, $\neg (x E_{G} y_{n})$ by the definition of $E_{G}$. Fix any $y_{i}\neq y_{n}$ such that $y_{i}\in Y$. Since $\neg (y_{i} E_{G} y_{n})$, we have either $y_{i}\overline{R} y_{n}$ or $y_{n}\overline{R}y_{i}$. Since $y_{n}$ has the greatest length, we have $\neg(y_{n}\overline{R}y_{i})$. Thus $y_{i} \overline{R} y_{n}$. Since $\overline{R}$ is transitive, we have $y_{i}\overline{R}x$. Thus, $Y\cup\{x\}$ is an independent set which contradicts the maximality of $Y$.
\end{proof}

Then there exists an infinite independent set $I$ in $G$. Otherwise, if all independent sets in $G$ are finite, then there exists a maximal independent set $M$ in $G$ by (3). This contradicts claim \ref{claim 3.6} since $M$ has to be finite.  

Now, $I$ is countably infinite. Define $f:I\rightarrow\mathbb{N}$ such that if $y\in I$, then $f$ maps $y$ to $\vert y\vert$. Since $f$ is injective by the definition of $\overline{R}$, we are done.
Let $I=\{y_{n_{i}}:i\in\omega\}$ be an enumeration such that $\vert y_{n_{i}}\vert<\vert y_{n_{j}}\vert$ if $i<j$. 
Then 
$\{y_{n_{1}}^{\vert y_{n_{1}}\vert}, y_{n_{2}}^{\vert y_{n_{1}}\vert+1}, ..., y_{n_{2}}^{\vert y_{n_{2}}\vert}, y_{n_{3}}^{\vert y_{n_{2}}\vert+1},..., y_{n_{3}}^{\vert y_{n_{3}}\vert}, y_{n_{4}}^{\vert y_{n_{3}}\vert+1}$,...$\}$
is the desired infinite sequence, which guarantees that $\mathsf{DC}$ holds.
\end{proof}

\section{Rainbow Ramsey theorem, CRT, and weak choice forms}
\begin{thm}\label{Theorem 4.1}
{\em Fix any integers $n,m\geq 2$. The following hold:
\begin{enumerate}
    \item $\mathsf{RRT}^{m}_{n}$ holds in $\mathcal{N}_{1}$. Thus, $\mathsf{RRT}^{m}_{n}$ 
    does not imply “There are no amorphous sets” in $\mathsf{ZFA}$.
    
    \item $\mathsf{RRT}^{m}_{n}$ holds in $\mathcal{N}_{3}$.  Consequently, $\mathsf{RRT}^{m}_{n}$ implies none of $\mathsf{WOAM}$, $\mathsf{CS}$, and $\mathsf{A}$ (Antichain Principle) in $\mathsf{ZFA}$.
    
    \item $\mathsf{RRT}^{m}_{n} + \Delta SL$ holds in $\mathcal{N}_{26}$.
\end{enumerate}
} 
\end{thm}

\begin{proof}
(1). We recall the description of the basic Fraenkel model $\mathcal{N}_{1}$. We start with a model $M$ of $\mathsf{ZFA} + \mathsf{AC}$ where $A$ is a countably infinite set of atoms.
Let $\mathcal{G}$ be the group of all permutations of $A$, and $\mathcal{F}$ be the filter of subgroups of $\mathcal{G}$ generated by $\{$fix$_{\mathcal{G}}(E):E\in [A]^{<\omega}\}$.
The permutation model determined by $M$, $\mathcal{G}$, and $\mathcal{F}$ is the model $\mathcal{N}_{1}$. We recall that the set $A$ of atoms is amorphous in $\mathcal{N}_{1}$(see \cite{HR1998}). We show that if $X \in \mathcal{N}_{1}$ is an infinite set, then $\mathsf{RRT}^{m}_{n}$ holds for $X$. If $X$ is well-orderable in $\mathcal{N}_{1}$, then we are done by Fact \ref{Fact 3.1}(3). Assume that $X$ is non-well-orderable.
We recall the following result.

\begin{lem}\label{Lemma 4.2}(Blass \cite[Lemma, p.389]{Bla1977})
{\em Any non-well-orderable set in $\mathcal{N}_{1}$ has an 
infinite subset in one-to-one correspondence with a subset of A.} 
\end{lem}

Thus by Lemma \ref{Lemma 4.2}, $X$ contains a copy of an infinite subset $A'$ of $A$, call this copy $Y$. Without loss of generality, we assume $Y\subset A$ and $Y=A'$.

\begin{claim}
{\em $\mathsf{RRT}^{m}_{n}$ holds for $A'$ in $\mathcal{N}_{1}$ i.e., if $\chi: [A']^{m}\rightarrow C$ is an $n$-bounded coloring in $\mathcal{N}_{1}$, then there exists an infinite subset $A''$ of $A'$ such that $A''$ is polychromatic for $\chi$.} 
\end{claim}

\begin{proof}
Let $\chi$ be such a coloring and let $E\in [A]^{<\omega}$ be a support of $\chi$. Then $A'\backslash E$ is an infinite subset of $A'$. We show that $A'\backslash E$ is polychromatic for $\chi$. Otherwise, for some $P,Q\in[A'\backslash E]^{m}$ such that $P\backslash Q\neq \emptyset$, we have $\chi(P)=\chi(Q)$. Pick some $p\in P\backslash Q$. Let,
\begin{enumerate}
\item $\pi_{1}\in$ fix$_{\mathcal{G}}((E\cup P\cup Q)\backslash \{p\})$ be such that 
$\pi_{1}(p)=l_{1}$ for some $l_{1}\not\in E\cup P\cup Q$.
\item Suppose the sequences $\langle\pi_{i}:1\leq i\leq r\rangle$ and $\langle l_{i}:1\leq i\leq r\rangle$ are defined for some fixed $1\leq r< n$. Let $\pi_{r+1}\in$ fix$_{\mathcal{G}}((E\cup P\cup Q)\backslash \{p\})$ be such that 
$\pi_{r+1}(p)=l_{r+1}$ for some $l_{r+1}\not\in E\cup P\cup Q\cup \{l_{i}:1\leq i\leq r\}$. 
\end{enumerate}

Then for all $1\leq r\leq n$, $\pi_{r}(\chi(P))=\pi_{r}(\chi(Q))$. Thus,
$\chi(\{\pi_{r}(p):p\in P\})=\chi(\{\pi_{r}(q):q\in Q\})=\chi(Q)$, 
which contradicts the fact that $\chi$ is $n$-bounded.   
\end{proof}

Thus $\mathsf{RRT}^{m}_{n}$ holds for $Y$. By Fact \ref{Fact 3.1}(3), $\mathsf{RRT}^{m}_{n}$ holds for $X$.

(2). We recall the definition of Mostowski's
linearly ordered model $\mathcal{N}_{3}$  from \cite{HR1998}. We start with a model $M$ of $\mathsf{ZFA + AC}$ with a countably infinite set $A$ of atoms
using an ordering $<$ on $A$ chosen so that $(A, <)$ is order-isomorphic to the set $\mathbb{Q}$ of the rational numbers with the usual ordering. Let $\mathcal{G}$ be the group of all
order automorphisms of $(A, <)$ and $\mathcal{F}$ be the normal filter on $\mathcal{G}$ generated
by the subgroups $\{\text{fix}_{\mathcal{G}}(E), E \in [A]^{<\omega}\}$. Let $\mathcal{N}_{3}$ be the Fraenkel--Mostowski model
determined by $M$, $\mathcal{G}$, and $\mathcal{F}$. In order to show that $\mathsf{RRT}^{m}_{n}$ holds in $\mathcal{N}_{3}$, we will need the following result. 
\begin{lem}\label{Lemma 4.4}(Howard--Saveliev--Tachtsis \cite{HST2016})
{\em Any non-well-orderable set in $\mathcal{N}_{3}$ contains a copy of a bounded open interval of $A$.} 
\end{lem}

Let $X \in \mathcal{N}_{3}$ be an infinite set. In view of the arguments of (1), we may assume that $X$ is not well-orderable. Then by Lemma \ref{Lemma 4.4}, $X$ contains a copy of a bounded open interval of $A$, call this copy $Y$. Without loss of generality, we may assume that $Y\subset A$ and $Y=(a,b)$. 

\begin{claim}\label{claim 4.5}
{\em $\mathsf{RRT}^{m}_{n}$ holds for $Y$ in $\mathcal{N}_{3}$.}   
\end{claim} 

\begin{proof}
Let $\chi: [Y]^{m} \rightarrow C$ be a $n$-bounded coloring in $\mathcal{N}_{3}$ and let $E\in [A]^{<\omega}$ be a support of $\chi$.

Case(i): Suppose $E\cap Y\neq \emptyset$.
Let $e=max_{<}(E\cap Y)$ (we recall that the ordering $<$ of $A$ is in $\mathcal{N}_{3}$). Then $(e,b)$ is an infinite subset of $Y$. We show that $(e,b)$ is polychromatic for $\chi$. Otherwise, there exists $p_{1},...,p_{m},q_{1},...,q_{m}\in (e,b)$ as in (1) such that $\chi(\{p_{1},...,p_{m}\})=\chi(\{q_{1},...,q_{m}\})$. Without loss of generality, assume $p_{1}<...<p_{m}$ and $q_{1}<...<q_{m}$ with respect to the ordering of $A$ (for other orderings of $p_{i}$'s and $q_{i}$'s, the argument below will be similar). Let $P=\{p_{1},...,p_{m}\}$, and $Q=\{q_{1},...,q_{m}\}$. We follow the steps below:
\begin{enumerate}[$(a)$.]
    \item Let $s=max_{<}\{i:p_{j}=q_{j} (\forall 1\leq j\leq i)\}$. 
    \item Let $g=min_{<}((P\cup Q)\backslash \{p_{i}:1\leq i\leq s\})$ and $h=min_{<}((P\cup Q)\backslash (\{p_{i}:1\leq i\leq s\}\cup \{g\}))$. 
    \item Pick $l_{1}\in (g,h)$.
    \item Let $\pi_{1}\in$ fix$_{\mathcal{G}}((E\cup P\cup Q)\backslash \{g\})$ be an ordered automorphism such that $\pi_{1}(g)=l_{1}$.

    \item If the sequence of ordered automorphisms $\langle\pi_{1}, ..., \pi_{r}\rangle$ and the sequence $\langle l_{1}, ..., l_{r}\rangle$ are defined for some $1\leq r< n$, then pick $l_{(r+1)}\in (l_{r},h)$ and let $\pi_{r+1}\in$ fix$_{\mathcal{G}}((E\cup P\cup Q)\backslash \{g\})$ be an ordered automorphism such that $\pi_{r+1}(g)=l_{(r+1)}$. 
\end{enumerate}
Then we obtain a sequence of ordered automorphisms $\langle\pi_{1}, ..., \pi_{n}\rangle$ such that $\pi_{i}(g)\neq \pi_{j}(g)$ for all $i\neq j$ and $1\leq i,j\leq n$.
The rest follows by the arguments of (1).

Case(ii): Suppose $E\cap Y= \emptyset$. 
Then by the arguments in Case (i), $(a,b)$ is an infinite set which is polychromatic for $\chi$. 
\end{proof}
Thus by Fact \ref{Fact 3.1}(3), $\mathsf{RRT}^{m}_{n}$ holds for $X$. The rest follows by the following known facts about $\mathcal{N}_{3}$:

\begin{enumerate}[(i)]
    \item $\mathsf{WOAM}$ fails in $\mathcal{N}_{3}$ (cf. \cite{HR1998}).
    \item $\mathsf{CS}$ and $\mathsf{LW}$ (Every linearly ordered set can be well ordered) fail in $\mathcal{N}_{3}$ \cite[Theorem 7]{Tac2022} and $\mathsf{A}$ implies $\mathsf{LW}$ in $\mathsf{ZFA}$ \cite[Theorem 9.1]{Jec1973}.
\end{enumerate}

(3). We recall the description of  Brunner/Pincus’s Model $\mathcal{N}_{26}$. We start
with a model $M$ of $\mathsf{ZFA} + \mathsf{AC}$ with a denumerable set $A$ of atoms which is
a denumerable disjoint union of denumerable sets, so that $A =\bigcup \{P_{n}: n \in \omega\}$, where $\{P_{n}: n \in \omega\}$ is disjoint and $\vert P_{n}\vert = \aleph_{0}$ for all $n \in \omega$. Let $\mathcal{G}$ be the group of all permutations $\phi$ of $A$ such that $\phi(P_{n}) = P_{n}$ for all $n \in \omega$.
Let $\mathcal{I}$ be the normal ideal of all finite subsets of $A$. Then,
$\mathcal{N}_{26}$ is the permutation model determined by $M$, $\mathcal{G}$ and $\mathcal{I}$. Since for every $n \in \omega$, $P_{n}$ is amorphous in $\mathcal{N}_{26}$ (see \cite{HR1998}), following the arguments of Blass \cite[Lemma, p.389]{Bla1977} and Tachtsis \cite[Lemma 4]{Tac2018}, we can see the following:

\begin{claim}\label{claim 4.6}
{\em The following hold in $\mathcal{N}_{26}$:
\begin{enumerate}
    \item Any non-well-orderable set contains a copy of an infinite subset of $P_{n}$ for some $n \in \omega$.

    \item Any non-well-orderable family of non-empty sets has a non-well-orderable subfamily with a choice function.
\end{enumerate}
}
\end{claim}

By claim \ref{claim 4.6}(1) and following the arguments of (1), $\mathsf{RRT}^{m}_{n}$ holds in $\mathcal{N}_{26}$.
By applying claim \ref{claim 4.6}(2) and following the arguments due to Tachtsis \cite[proof of Theorem 7(i)]{Tac2018}, we can see that $\mathsf{PC}$ is true in $\mathcal{N}_{26}$. For the reader's convenience, we write down the proof. Let $\mathcal{F}$ be an uncountable family (i.e. $\vert \mathcal{F}\vert \not\leq \aleph_{0}$) of non-empty countable sets in $\mathcal{N}_{26}$. If $\mathcal{F}$ is well-orderable, then the conclusion follows from the fact that in $\mathcal{N}_{26}$, $\mathsf{WUT}$ holds (see \cite{HR1998}). If $\mathcal{F}$ is not well-orderable,
then by claim \ref{claim 4.6}(2), there is a non-well-orderable and thus an uncountable subfamily $\mathcal{F}'$ of $\mathcal{F}$ with a choice function in $\mathcal{N}_{26}$. The rest follows by Fact \ref{Fact 3.1}(10) and the fact that $\mathsf{CUT}$ holds in $\mathcal{N}_{26}$.
\end{proof}

\begin{remark}\label{Remark 4.7}
We can observe a different argument to see that $\mathsf{RRT}^{2}_{n}$ holds in $\mathcal{N}_{3}$. Tachtsis \cite{Tac2016} proved that $\mathsf{RT}^{2}_{2}$ (and thus $\mathsf{RT}^{2}_{n}$ by Fact \ref{Fact 3.1}(8)) holds in $\mathcal{N}_{3}$. Since $\mathsf{AC}_{fin}$ holds in $\mathcal{N}_{3}$,\footnote{In fact, for every family of non-empty well-orderable sets there exists a choice function in $\mathcal{N}_{3}$ (see \cite[Chapter 4, Problem 14, page 53]{Jec1973})).} we have that $\mathsf{AC}_{\leq n}$ holds in $\mathcal{N}_{3}$. The rest follows by Galvin's trick mentioned in Fact \ref{Fact 3.1}(1).
\end{remark}

\begin{thm}\label{Theorem 4.9}
{\em The following hold:
\begin{enumerate}
    \item $\mathsf{OP}+\mathsf{RT}^{4}_{203}$ implies $\mathsf{CRT}$ in $\mathsf{ZF}$.
    \item $\mathsf{DF=F}$ is strictly stronger than $\mathsf{CRT}$ in $\mathsf{ZFA}$.
\end{enumerate}
}
\end{thm}

\begin{proof}
(1). Let $X$ be an infinite set, and $\chi:[X]^{2}\rightarrow C$ be a coloring. By $\mathsf{OP}$, $X$ is linearly orderable. Fix a linear ordering $<$ of $X$. Consider the coloring $\chi':[X]^{4}\rightarrow 203$ as in Proposition \ref{Proposition 3.2}(1), with respect to the ordering $<$ of $X$. By $\mathsf{RT}^{4}_{203}$, there exists an infinite set $Y\subseteq X$ which is monochromatic for $\chi'$. Define $R=<\restriction Y$. Since $R$ is a linear ordering on $Y$, following the proof of Proposition \ref{Proposition 3.2}(1) (in particular, following the arguments in \cite[the proof of Theorem 2 in section 5.5, page 129]{GRS1980}), we can see that $Y$ is the desired infinite subset of $X$ with ordering $R$.

(2). We show that $\mathsf{CRT}$ holds in Mostowski’s linearly ordered model $\mathcal{N}_{3}$ where $\mathsf{DF=F}$ fails (see \cite{HR1998}). Working in much the same way as in Tachtsis \cite[Theorem 2.4]{Tac2016}, we can see that $\mathsf{RT}^{4}_{203}$ holds in $\mathcal{N}_{3}$. Since $\mathsf{OP}$ is true in $\mathcal{N}_{3}$, the rest follows from (1) and Proposition \ref{Proposition 3.2}(1).
\end{proof}




\begin{remark}
We remark that {\em $\mathsf{AC^{WO}_{fin}}+\mathsf{RRT}^{m}_{n}+\mathsf{CRT}$ does not imply $\mathsf{CAC}^{\aleph_{0}}$ in $\mathsf{ZFA}$}.
Consider the model $\mathcal{N}_{41}$ from \cite{HR1998}. 
We start with a model $M$ of $\mathsf{ZFA + AC}$ where $A =\bigcup \{A_{n}:n \in \omega\}$ is a disjoint union, where each $A_{n}$ is countably infinite  and for each $n \in\omega$, 
$(A_{n}, \leq_{n}) \cong (\mathbb{Q}, \leq)$ (i.e., ordered like the rationals by $\leq_{n}$). Let $\mathcal{G}$ be the group of all permutations on $A$ such that for all $n \in \omega$, and all $\phi\in \mathcal{G}$, $\phi$ is an order automorphism of $(A_{n}, \leq_{n})$. Let $\mathcal{I}$ be the normal ideal of subsets of $A$ which is generated by finite unions of $A_{n}$’s. Let $\mathcal{N}_{41}$ be the Fraenkel–Mostowski model determined by $M$, $\mathcal{G}$, and $\mathcal{I}$.
In $\mathcal{N}_{41}$, $\mathsf{DF=F}$ holds  (cf. \cite[Theorem 4]{Tac2019b}, \cite[Note 112]{HR1998}).
Consequently, $\mathcal{N}_{41}\models\mathsf{RRT}^{m}_{n}+\mathsf{CRT}$ by Proposition \ref{Proposition 3.2}(1,2). In \cite[the proof of Theorem 4.1(4)]{BG2023}, Banerjee and Gopaulsingh observed that $\mathsf{CAC}^{\aleph_{0}}$ fails in $\mathcal{N}_{41}$. 

In \cite[Note 112]{HR1998}, it is shown that $\mathsf{AC^{WO}_{fin}}$ is true in $\mathcal{N}_{41}$. We present a different argument that is fairly similar to the one given in \cite{HR1998}. In particular, we follow the methods of Tachtsis \cite[Theorem 4]{Tac2019b} to show that $\mathsf{AC^{WO}_{fin}}$ is true in $\mathcal{N}_{41}$ which uses the
following result of Truss.

\begin{lem}{(Truss \cite[Theorem 3.5]{Tru1989})}\label{Lemma 4.10}
{\em Let $A(\mathbb{Q})$ be the group of all order automorphisms of $(\mathbb{Q}, \leq)$.
If $H$ is a subgroup of $A(\mathbb{Q})$ of index less than $2^{\aleph_{0}}$, then for some finite $A \subset \mathbb{Q}$,
{\em fix}$(A) = H$ (i.e., $H = \{\phi \in A(\mathbb{Q}) : (\forall a \in A)(\phi(a) = a)\})$. Thus, every proper
subgroup of $A(\mathbb{Q})$ has  infinite index in $A(\mathbb{Q})$.} 
\end{lem}
Let $\mathcal{X}=\{X_{\alpha}:\alpha\in\kappa\}$ be an infinite well-ordered set of non-empty finite sets in $\mathcal{N}_{41}$ for some infinite well-ordered cardinal $\kappa$. Let $E =\bigcup\{A_{n} : n < k + 1\}$ (where $k \in \omega$) be a support of $X_{\alpha}$ for all $\alpha\in\kappa$. 
Such an $E$ exists by Lemma \ref{Lemma 2.4}, since 
$\mathcal{X}$ is well orderable in $\mathcal{N}_{41}$.
We show that $E$ is a support of every element of $\bigcup\mathcal{X}$; hence,
$\bigcup\mathcal{X}$ will be well orderable in the model by Lemma \ref{Lemma 2.4}. Otherwise,
there exists $\alpha\in\kappa$ and  $x \in X_{\alpha}$ such that $E$ is not a support of $x$.
Then there exists a permutation $\phi \in $ fix$_{\mathcal{G}}(E)$ such that $\phi(x) \neq x$. Let $E_{x}$ be a support of $x$. Without loss of generality, assume that $E_{x} = E \cup A_{k+1}$ and $\phi \in$ fix$_{\mathcal{G}}(A\backslash A_{k+1})$. We follow the ideas of  Tachtsis \cite[Theorem 4]{Tac2019b} to observe the following:

\begin{enumerate}
\item  The group fix$_{\mathcal{G}}(A\backslash A_{k+1})$ is isomorphic
to $A(A_{k+1})$.  We
denote fix$_{\mathcal{G}}(A\backslash A_{k+1})$ by $\mathcal{G}'$.
\item The $\mathcal{G}'$-orbit of $x$, $Orb_{\mathcal{G}'}(x)=\{\phi(x):\phi\in \mathcal{G}'\}$ is a subset of $X_{\alpha}$ as $x\in X_{\alpha}$, $E$ is a support of $X_{\alpha}$, and $\mathcal{G}'\subseteq$ fix$_{\mathcal{G}}(E)$. Hence, $Orb_{\mathcal{G}'}(x)$ is finite.
\item Since $\phi(x),x\in X_{\alpha}$ and $\phi(x) \neq x$, we have $\vert Orb_{\mathcal{G}'}(x)\vert\geq 2$. Moreover, $Stab_{\mathcal{G}'}(x) = \{\eta \in \mathcal{G}' : \eta(x) = x\}$ is a proper subgroup of $\mathcal{G}'$. 
\item Since $Orb_{\mathcal{G}'}(x)$ is finite, the index $\vert \mathcal{G}':Stab_{\mathcal{G}'}(x)\vert$ is finite. Thus, $Stab_{\mathcal{G}'}(x)$ is a proper subgroup of $A({A}_{k+1})$ which has finite index in $A({A}_{k+1})$. This contradicts Lemma \ref{Lemma 4.10}.
\end{enumerate}
\end{remark}

\section{Results in the Basic Cohen Model}
\begin{defn}\label{Definition 5.1} We recall the description of the Basic Cohen Model $\mathcal{M}_{1}$.
Let $M$ be a countable transitive model of $\mathsf{ZF} + \mathsf{V = L}$. The notion of forcing is the poset $\mathbb{P}=(P, \leq)$ (of $M$), where $P$ is the set of all finite partial functions from
$\omega \times \omega$ into $2$, and $\leq$ is reverse inclusion, i.e., $p \leq q$ if and only if
$p \supseteq q$. Let $G$ be a $\mathbb{P}$-generic set over $M$ and let $M[G]$ be the generic extension model of $M$. Every permutation $\pi$ of $\omega$ induces an order automorphism
of $(P, \leq)$ as follows:
\begin{center}
$dom(\pi p) = \{(\pi n, m):(n, m) \in dom(p)\},
(\pi p)(\pi n, m) = p(n, m)$.    
\end{center}

Let $\mathcal{G}$ be the group of all order automorphisms of $(P, \leq)$ induced by
permutations of $\omega$ as above and $\mathcal{F}$ be the normal filter on $\mathcal{G}$ generated by $\{$fix$_{\mathcal{G}}(E), E \in [\omega]^{<\omega}\}$. 
For a $\mathbb{P}$-name $\tau$, we denote its symmetric group with respect to $\mathcal{G}$ by $sym^{\mathcal{G}}(\tau) =\{g\in \mathcal{G} : g\tau = \tau\}$ and say $\tau$ is {\em symmetric} with respect to $\mathcal{F}$ if $sym^{\mathcal{G}}(\tau)\in\mathcal F$. Let $HS$ be the class of all hereditary symmetric names i.e.,  for a $\mathbb{P}$-name $\tau$, 
    $\tau\in HS$ iff $\tau$ is symmetric with respect to $\mathcal{F}$, and for each $\sigma\in dom(\tau)$, $\sigma\in HS$.
The symmetric submodel $\{\tau^{G} : \tau \in HS\}$ of $M[G]$ is the model $\mathcal{M}_{1}$. We refer the reader to \cite[Section 5.3]{Jec1973} for details concerning $\mathcal{M}_{1}$.
Let $A =\{x_{n}: n \in \omega\}$ be the  countably many added Cohen reals (where for $n \in \omega$, $x_{n} = \{m \in \omega : \exists p \in G, p(n, m) = 1\}$) in $\mathcal{M}_{1}$. 
\end{defn}
In \cite{Pal2013}, Palumbo proved the following lemma and used it to show that $\mathsf{RRT}^{2}_{2}$ holds in $\mathcal{M}_{1}$.

\begin{lem}\label{Lemma 5.2}{(Palumbo; \cite[Lemma 2.6]{Pal2013})} {\em If $B \in \mathcal{M}_{1}$ is a non-well-orderable set, then $\mathcal{M}_{1}$ contains a bijection of $B$ with an infinite subset of $A$.}
\end{lem}

Inspired by the proof of \cite[Lemma 4]{Tac2018} due to Tachtsis (see Fact 3.1(9)), we incorporate the arguments from \cite[proof of Lemma 2.6]{Pal2013} to observe the following lemma which we need to prove Theorem \ref{Theorem 5.8}.

\begin{lem}\label{Lemma 5.3}
{\em In $\mathcal{M}_{1}$, every non-well-orderable family of non-empty sets has a non-well-orderable subfamily with a choice function.}
\end{lem}

\begin{proof}
The following notations and facts will be useful to our proof (see \cite[proof of Lemma 2.6]{Pal2013} and \cite[section 5.5, page 72]{Jec1973}).
\begin{enumerate}
    \item Let $\dot{x}\in HS$ be a  hereditary symmetric name and let $e\in[\omega]^{<\omega}$. We say that $e$ is a {\em support} of $\dot{x}$ if $sym_{\mathcal{G}} (\dot{x})\supseteq$ fix$_{\mathcal{G}}(e)$.
    \item Fix $E = \{x_{n_{1}}, ..., x_{n_{k}}\}\in[A]^{<\omega}$. We say that the canonical name $\underline{E}$ of $E$ (defined in \cite[section 5.5, page 72]{Jec1973}) is a support of $\dot{x}$ whenever $\{n_{1}, ..., n_{k}\}$ supports $\dot{x}$.
    \item Let $x\in \mathcal{M}_{1}$. We say $E=\{x_{i_{0}},...,x_{i_{k}}\}\in [A]^{<\omega}$ is a support of $x$ (and write $\Delta(E,x)$) if there exists $\dot{x}\in HS$ with $\dot{x}^{G}=x$ such that 
    $\underline{E}$ is a support of $\dot{x}$, i.e.,
    $\{i_{0},...,i_{k}\}$ is a support of $\dot{x}$ by (2). 
    \item If $B\in\mathcal{M}_{1}$ and there exists $E\in [A]^{<\omega}$ such that 
    $\Delta(E,x)$ holds
    for all $x\in B$, then $B$ is well-orderable in $\mathcal{M}_{1}$ (see \cite[proof of Lemma 2.6]{Pal2013}).
    \item Every $x\in \mathcal{M}_{1}$ has a least support (see \cite[Lemma 5.22]{Jec1973}).
    \item The set of Cohen reals $A$ is Dedekind-finite in $\mathcal{M}_{1}$.
\end{enumerate}

Let $\mathcal{F}$ be a non-well-orderable family of non-empty sets in $\mathcal{M}_{1}$, and $E_{0}=\{x_{i_{0}},...,x_{i_{l_{1}}}\}$ 
be the least support of $\mathcal{F}$ by item (5).
By item (4), there exists an element $x \in \mathcal{F}$ such that $\Delta(E_{0},x)$ does not hold (i.e., $E_{0}$ is not a support of $x$). By item (5), let $E_{1}\cup \{x_{k}\}$ be the least support of $x$ for some $x_{k} \not\in E_{0} \cup E_{1}$. Let $\{x_{j_{0}},...,x_{j_{l_{2}}}\}$ be the enumeration of $E_{1}$. 

Fix $y\in x$, and let $E_{y}\supseteq E_{0}\cup E_{1}\cup \{x_{k}\}$ be a support of $y$. 
Then by item (3), there exists $\dot{x},\dot{y}\in HS$ with $\dot{x}^{G}=x$ and $\dot{y}^{G}=y$ such that 
$\underline{E_{1}\cup \{{x}_{k}}\}$ is the least support of $\dot{x}$ and $\underline{E_{y}}$ is the support of $\dot{y}$. 
Define $F=E_{y}\backslash \{x_{k}\}$. 
Let $p\in G$ be a forcing condition such that the following holds:
\begin{enumerate}
    \item $p\Vdash `\dot{x}\in \dot{\mathcal{F}}$',
    \item $p\Vdash `\underline{E_{1}\cup \{{x}_{k}}\}$ is the least support of $\dot{x}$',  
    \item $p\Vdash `\dot{y}\in \dot{x}$',
    \item $p\Vdash `\underline{E_{y}}$ is a support of $\dot{y}$'.
\end{enumerate}
In \cite[Lemma 2.6]{Pal2013}, Palumbo proved that if $\sigma=\{\langle\langle\pi(\dot{x_{k}}),\pi(\dot{x})\rangle, \pi(p) \rangle: \pi\in$ fix$_{\mathcal{G}}(e)\}$ where $e=e_{0}\cup e_{1}$ such that $e_{0}$ is the indices of $E_{0}$, and $e_{1}$ is the indices of $E_{1}$, then 
\begin{center}
$f=\sigma^{G}=\{\langle \pi(\dot{x_{k}})^{G}, \pi(\dot{x})^{G}\rangle: \pi\in $ fix$_{\mathcal{G}}(e), \pi(p)\in G\}$     
\end{center}
is an injection in $\mathcal{M}_{1}$ such that $ran(f)\subseteq \mathcal{F}$ and $dom(f)$ is an infinite subset of $A$.\footnote{We denote $ran(f)$ by the range of $f$, and $dom(f)$ by the domain of $f$.} Define 
\begin{center}
$\sigma_{1}=\{\langle\langle\pi(\dot{x})
,\pi(\dot{y})\rangle, \pi(p) \rangle: \pi\in$ fix$_{\mathcal{G}}(g')\}$ 
\end{center}

where $g'$ is the indices of $F$. Then $g= \sigma_{1}^{G}=\{\langle \pi(\dot{x})^{G}, \pi(\dot{y})^{G}\rangle: \pi\in$ fix$_{\mathcal{G}}(g'),\pi(p)\in G\}$ is in $\mathcal{M}_{1}$ since $g'$ is a support of $\sigma_{1}$. 
We can slightly modify the arguments of Tachtsis \cite[Lemma 4]{Tac2018}, to see that $g$ is a choice function of the non-well-orderable subfamily $\mathcal{F}'=dom(g)$ of $\mathcal{F}$. We write the arguments for the reader's convenience. 
\begin{claim}
{\em The following hold:
\begin{enumerate}[(i)]
    \item $S=\{\pi(\dot{x}_{k})^{G}: \pi \in$ fix$_{\mathcal{G}}(g'), \pi(p)\in G\}$ is an infinite set.
    \item $\mathcal{F}'=dom(g)$ is a non-well-orderable set.
\end{enumerate}
}     
\end{claim}

\begin{proof}
(i). 
We can see that if $\pi_{i}(p) \in G$ and $\pi_{i}\in$ fix$_{\mathcal{G}}(g')$ such that $\pi_{i}(k)=i$ for some $i\in \omega\backslash \{g'\}$, then $\pi_{i}(\dot{x}_{k})^{G}\in S$. Since $G$ is the generic filter, there are infinitely many such distinct $i$ by a genericity argument.

(ii). Clearly, $\mathcal{F}'= dom(g) = \{\pi(\dot{x})^{G}: \pi \in$ fix$_{\mathcal{G}}(g'), \pi(p)\in G\}\subseteq \{\pi(\dot{x})^{G}: \pi \in$ fix$_{\mathcal{G}}(e), \pi(p)\in G\}=ran(f)\subseteq\mathcal{F}$ as $e\subset g'$. We show that $\mathcal{F}'$ is infinite.
Pick $\pi_{1},\pi_{2}\in$ fix$_{\mathcal{G}}(g')$ such that $\pi_{1}(p),\pi_{2}(p)\in G$ where $\pi_{1}(\dot{x}_{k})^{G}\neq \pi_{2}(\dot{x}_{k})^{G}$. Thus, $\pi_{1},\pi_{2}\in$ fix$_{\mathcal{G}}(e)$ as $e\subset g'$. Since $f$ is injective, we have $\pi_{1}(\dot{x})^{G}\neq \pi_{2}(\dot{x})^{G}$. Thus, $\mathcal{F}'$ is infinite as $S$ is infinite by (i).

If $\mathcal{F}'\subseteq ran(f)$ is an infinite well-orderable set, then $ran(f)$ is Dedekind-infinite. This contradicts the fact that $ran(f)$ is Dedekind-finite since $f$ is an injective function and $A$ is Dedekind-finite in $\mathcal{M}_{1}$ by item (6).   
\end{proof}

\begin{claim}
{\em $g$ is a function.}     
\end{claim}
\begin{proof}
Assume the contrary; then there exist $\phi,\rho \in$ fix$_{\mathcal{G}}(g')$ such that $\phi(p),\rho(p)\in G$, $\phi(\dot{x})^{G} = \rho(\dot{x})^{G}$, but $\phi(\dot{y})^{G} \neq \rho(\dot{y})^{G}$. 
We claim that $\phi(\dot{x})^{G} \neq \rho(\dot{x})^{G}$ to obtain a contradiction. 

We must have, $\phi(k)\neq\rho(k)$. Otherwise, $\phi^{-1}\rho\in$ fix$_{\mathcal{G}}(g'\cup\{k\})$.
Since $E_{y}$ is a support of $y$, we can see that the indices of $E_{y}$ i.e., $\{k\}\cup g'$, is a support of $\dot{y}$ by item (3).
Thus $\phi^{-1}\rho$ fixes $\dot{y}$, and so $\phi(\dot{y})^{G}=\rho(\dot{y})^{G}$, which is a contradiction. 

Since $\phi(k)\neq\rho(k)$, we have $\phi(\dot{x}_{k})^{G}\neq\rho(\dot{x}_{k})^{G}$.\footnote{see \cite[proof of Lemma 2.6, page 955, 3rd line]{Pal2013}.} Thus since $\phi,\rho\in $ fix$_{\mathcal{G}}(e)$ (as $e\subset g'$ and $\phi,\rho \in$ fix$_{\mathcal{G}}(g')$), and $f$ is an injective function, we have $f(\phi(\dot{x}_{k})^{G})\neq f(\rho(\dot{x}_{k})^{G})$, i.e., $\phi(\dot{x})^{G}\neq\rho(\dot{x})^{G}$. 
\end{proof}

\begin{claim}
{\em In $\mathcal{M}_{1}$, $g$ is a choice function of the non-well-orderable subfamily $\mathcal{F}'=dom(g)$ of $\mathcal{F}$.}
\end{claim}

\begin{proof}
Since $p\Vdash \dot{y}\in \dot{x}$, we have $\pi(p)\Vdash \pi(\dot{y})\in \pi(\dot{x})$ for all $\pi\in$ fix$_{\mathcal{G}}(g')$ such that $\pi(p)\in G$. Thus $\pi(\dot{y})^{G}\in \pi(\dot{x})^{G}$ in $\mathcal{M}_{1}$. Consequently, $g$ is a choice function of $\mathcal{F}'$ in $\mathcal{M}_{1}$.    
\end{proof}
\end{proof}

We modify the arguments of Palumbo \cite[Lemma 2.5]{Pal2013}, to observe the following lemma which we need to prove Theorem \ref{Theorem 5.8}.

\begin{lem}\label{Lemma 5.7}
{\em If $Y$ is an infinite subset of $A$ in $\mathcal{M}_{1}$ and if $\chi: [Y]^{m} \rightarrow C$ is an $n$-bounded
coloring in $\mathcal{M}_{1}$, then there is an infinite set $X \subseteq Y$ in $\mathcal{M}_{1}$ such that $Y$ is polychromatic for $\chi$.}    
\end{lem}

\begin{proof}
Let $\dot{\chi}$ and $\dot{Y}$ be hereditarily symmetric names for $\chi$ and $Y$. Since $\chi\in\mathcal{M}_{1}$, there exists $e\in [\omega]^{<\omega}$ such that fix$_{\mathcal{G}}(e)\subseteq sym(\dot{\chi})$. We claim that $Y\backslash\{x_{n}:n\in e\}$ is polychromatic for $\chi$. 
Otherwise, for some $I,J\in [Y\backslash\{x_{n}:n\in e\}]^{m}$ such that $I\backslash J\neq\emptyset$, we have $\chi(I)=\chi(J)$.
Let $I=\{x_{i_{1}},...,x_{i_{m}}\}$, $J=\{x_{j_{1}},...,x_{j_{m}}\}$, $I'=\{i_{1},...,i_{m}\}$, and $J'=\{j_{1},...,j_{m}\}$.
Let $p\in \mathbb{P}$ be such that $p\Vdash \dot{\chi}$ is $n$-bounded and $p\Vdash \dot{\chi}(\{\dot{x}_{j_{1}},...,\dot{x}_{j_{m}}\})=\dot{\chi}(\{\dot{x}_{i_{1}},...,\dot{x}_{i_{m}}\})$. 
Pick some element from $I\backslash J$, say $x_{i_{k}}$.
Without loss of generality, assume that $e\cup I'\cup J'\subseteq dom(p)$. Let,
\begin{enumerate}
\item $\pi_{1}\in$ fix$_{\mathcal{G}}((e\cup I'\cup J')\backslash \{i_{k}\})$ be such that 
$\pi_{1}(i_{k})=l_{1}$ for some $l_{1}\not\in dom(p)$.
\item Suppose the sequences $\langle\pi_{i}:1\leq i\leq r\rangle$ and $\langle l_{i}:1\leq i\leq r\rangle$ are defined for some fixed $1\leq r< n-1$. Let $\pi_{r+1}\in$ fix$_{\mathcal{G}}((e\cup I'\cup J')\backslash \{i_{k}\})$ be such that 
$\pi_{r+1}(i_{k})=l_{r+1}$ for some $l_{r+1}\not\in dom(p)\cup \{l_{i}:1\leq i\leq r\}$. 
\end{enumerate}

Fix any $1\leq r\leq n-1$. Since $p$ and $\pi_{r}(p)$ are compatible conditions,  
$\pi_{r}(p)\Vdash \dot{\chi}(\{\dot{x}_{j_{1}},...,\dot{x}_{j_{m}}\})=\dot{\chi}(\{\dot{x}_{i_{1}},...,\dot{x}_{i_{k-1}}, \dot{x}_{l_{r}},
\dot{x}_{i_{k+1}},...,\dot{x}_{i_{m}}\})$.
Then for each $1\leq r\leq n-1$,
\begin{center}
$p\cup (\bigcup_{1\leq r\leq n-1}\pi_{r}(p)) \Vdash \dot{\chi}(\{\dot{x}_{i_{1}},...,\dot{x}_{i_{m}}\})=\dot{\chi}(\{\dot{x}_{j_{1}},...,\dot{x}_{j_{m}}\}) =\dot{\chi}(\{\dot{x}_{i_{1}},...,\dot{x}_{i_{k-1}}, \dot{x}_{l_{r}},
\dot{x}_{i_{k+1}},...,\dot{x}_{i_{m}}\})$
\end{center} 
as well as $p\cup (\bigcup_{1\leq r\leq n-1}\pi_{r}(p))\Vdash$ ``$\chi$ is $n$-bounded'' which is a contradiction.
\end{proof}

\begin{thm}\label{Theorem 5.8}{($\mathsf{ZF}$)}
{\em Fix any $n,m,k,l\in \omega\backslash\{0,1\}$. Then the following hold:
\begin{enumerate}
    \item $\Delta SL + \mathsf{RRT}^{m}_{n}$ implies neither $\mathsf{RT}^{k}_{l}$ nor $\mathsf{EDM}$.
    \item $\mathsf{DF=F}$ is strictly stronger than $\mathsf{RRT}^{m}_{n}$.
    \item $\mathsf{CRT}$ is strictly stronger than $\mathsf{RRT}^{2}_{n}$. 
\end{enumerate}    
}
\end{thm}

\begin{proof}By applying Lemma \ref{Lemma 5.7}, and following the arguments of Palumbo \cite[Theorem 2.4]{Pal2013} (specifically Lemma \ref{Lemma 5.2}), we can see that $\mathsf{RRT}^{m}_{n}$ holds in $\mathcal{M}_{1}$.

(1). Blass \cite{Bla1977} proved that $\mathsf{RT}^{2}_{2}$ fails in $\mathcal{M}_{1}$, and hence, $\mathsf{RT}^{k}_{l}$ fails by Fact \ref{Fact 3.1}(8). Banerjee--Gopaulsingh \cite{BG2023} observed that $\mathsf{EDM}$ fails in $\mathcal{M}_{1}$ since $\mathsf{EDM}$ implies $\mathsf{RT}^{2}_{2}$ in $\mathsf{ZF}$ (see \cite[Theorem 4.1(3)]{BG2023}).   
Since $\mathsf{WUT}$ holds in $\mathcal{M}_{1}$ (see \cite{HR1998}), $\mathsf{PC}$ is true in $\mathcal{M}_{1}$, by applying Lemma \ref{Lemma 5.3} and following the arguments in the proof of Theorem \ref{Theorem 4.1}(3).
By Fact \ref{Fact 3.1}(10), $\Delta SL$ holds in $\mathcal{M}_{1}$ as $\mathsf{CUT}$ holds in $\mathcal{M}_{1}$ (see \cite{HR1998}). 

(2-3). Since $\mathsf{RT}^{2}_{2}$ fails in $\mathcal{M}_{1}$, $\mathsf{CRT}$ fails as well by Proposition \ref{Proposition 3.2}(3). Since the set of Cohen reals $A$ is Dedekind-finite in $\mathcal{M}_{1}$, $\mathsf{DF=F}$ fails in $\mathcal{M}_{1}$. The rest follows from Proposition \ref{Proposition 3.2}(2,3).
\end{proof}

\section{Amorphous sets and Ramsey type theorems} 
\begin{defn}
Let $A$ be a  countably infinite set of atoms, $\mathcal{G}$ be the group of all permutations of $A$ which moves only finitely many atoms, and $S$ be the set of all finite partitions of $A$. In \cite[Lemma 4.1]{Bru2016}, Bruce sketched a proof of the fact that $\mathcal{F}$ = $\{H:$ $H$ is a subgroup of $\mathcal{G}$, $H \supseteq$ fix$_{\mathcal{G}}(P)$ for some $P \in S\}$ is a normal filter of subgroups of $\mathcal{G}$. The model $\mathcal{V}_{fp}$ is the permutation model determined by $M$, $\mathcal{G}$ and $\mathcal{F}$. Without loss of generality, we assume that the supports consist of infinite blocks and singletons only. 
\end{defn}

We modify the arguments of Blass \cite[Lemma, p.389]{Bla1977}, to observe the following lemma. 

\begin{lem}\label{Lemma 6.2}
{\em In $\mathcal{V}_{fp}$, every non-well-orderable set contains a copy of an infinite subset of $A$.}
\end{lem}

\begin{proof}
Let $X$ be a non-well-orderable set, and $P_{1}$ be a support of $X$. By Lemma \ref{Lemma 2.4}, there is an $x \in X$ and a $\eta \in \text{fix}_{\mathcal{G}}(P_{1})$ such that $\eta(x) \neq x$.  
Let $P_{2}$ be a support of $x$. Let $S(P_{1})$ and $S(P_{2})$ be the set of singletons blocks of $P_{1}$ and $P_{2}$ respectively.
Without loss of generality, we may assume that $P_{2}$ is a refinement of $P_{1}$ such that $\vert P_{2}\vert-\vert P_{1}\vert$ is least. 

Case(i): There exists an $\{a\}\in S(P_{2})\backslash S(P_{1})$. Pick any infinite block from $P_{2}$ (say $P_{2}^{1}$), and let $I(P_{2})$ be the set of infinite blocks of $P_{2}$. Let,
$P_{3}$ be a partition whose singleton blocks are from $S(P_{2}) \backslash \{a\}$
and the infinite blocks are all the blocks of
$(I(P_{2})\backslash \{P_{2}^{1}\})$ as well as $P_{2}^{1}\cup \{a\}$ i.e.,
\begin{center}
 $P_{3}=\{P_{2}^{1}\cup \{a\}\}\cup (P_{2}\backslash \{\{a\},P^{1}_{2}\})$.   
\end{center}
 
Define $f=\{(\pi(a),\pi(x)): \pi\in$ fix$_{\mathcal{G}}(P_{3})\}$. Clearly,  $P_{3}$ is a support of $f$. Let $\pi_{1}(a)=\pi_{2}(a)$ be such that $\pi_{1},\pi_{2}\in$ fix$_{\mathcal{G}}(P_{3})$. Then $\pi_{2}^{-1}\pi_{1}(a)=a$ and $\pi_{2}^{-1}\pi_{1}\in$ fix$_{\mathcal{G}}(P_{3})$, which implies that $\pi_{2}^{-1}\pi_{1}\in$ fix$_{\mathcal{G}}(P_{2})$. Since $P_{2}$ is a support of $x$, we have $\pi_{2}^{-1}\pi_{1}(x)=x$, and so $\pi_{1}(x)=\pi_{2}(x)$. Thus, $f$ is a well-defined function. 

We show that $f$ is injective. 
Otherwise, we can pick $\pi_{1},\pi_{2}\in$ fix$_{\mathcal{G}}(P_{3})$ such that $\pi_{1}(x)=\pi_{2}(x)$ and $\pi_{1}(a)\neq \pi_{2}(a)$. Let $\pi=\pi_{2}^{-1}\pi_{1}$. Then $\pi(x)=x$ but $\pi(a)\neq a$. Let $\pi(a)=b$. We show that $P_{3}$ is a support of $x$ to obtain a contradiction to the assumption that $P_{2}$ is a support of $x$ such that $\vert P_{2}\vert-\vert P_{1}\vert$ is least. Pick any $\tau\in$ fix$_{\mathcal{G}}(P_{3})$. We show that $\tau(x)=x$. Clearly, $(\tau (a),\tau(x))\in f$. Let $\sigma\in$ fix$_{\mathcal{G}}(P_{2})$ be such that $\sigma(b)=\tau(a)$. Since $\pi\in$ fix$_{\mathcal{G}}(P_{3})$, 
\begin{center}
    $(\pi(a),\pi(x))\in f\implies (\sigma\pi(a),\sigma\pi(x))\in \sigma(f)=f$ (as $\sigma\in$ fix$_{\mathcal{G}}(P_{3})$ and $P_{3}$ is a support of $f$).
\end{center}
Since, $P_{2}$ is a support of $x$ and $\sigma\in$ fix$_{\mathcal{G}}(P_{2})$, $\sigma(x)=x$. Since $\pi(x)=x$, we have $\sigma\pi(x)=\sigma(x)=x$. Moreover, $\sigma\pi(a)=\sigma(b)=\tau(a)$. Thus, $(\tau (a),x)\in f$. Consequently, $\tau(x)=x$ as $f$ is a well-defined function and $(\tau (a),\tau(x))\in f$.

Case(ii): There is no $\{a\}\in S(P_{2})\backslash S(P_{1})$ i.e., $S(P_{2})=S(P_{1})$. Since $P_{1}$ is not a support of $x$, $(\exists \eta\in$ fix$_{\mathcal{G}}(P_{1}))\eta(x)\neq x$. 
Since $\eta$ only moves finitely many atoms, it is decomposable into transpositions. Let $\eta=\prod_{1\leq i\leq n} f_{i}$ where $f_{i}$ is a transposition for each $1\leq i\leq n$. Thus there exists at least one $f_{i}$ say $(a,d)$, such that $(a,d)x\neq x$. We follow the steps below:
\begin{enumerate}
\item Let $P_{a}$ and $P_{d}$ be the blocks in $P_{3}$ which contains $a$ and $d$ respectively. 
If $P_{a}=P_{d}$, then $(a,d)\in$ fix$_{\mathcal{G}}(P_{2})$ and so $(a,d)x = x$ as $P_{2}$ is a support of $x$, which is a contradiction. Thus, $P_{a}\neq P_{d}$.
\item Let, $P_{3}$ be a partition whose blocks are $P_{a}\backslash \{a\}$, $P_{d}\cup \{a\}$, as well as all the blocks from $P_{2}\backslash\{P_{a},P_{d}\}$.
\item Consider $f=\{(\pi(a),\pi(x)): \pi\in$ fix$_{\mathcal{G}}(P_{3})\}$. Following the arguments in Case(i), $f$ is a well-defined function, and $P_{3}$ is a support of $f$. 

\item We can slightly modify the arguments of Case(i) to show that $f$ is injective. Assume $f$ is not injective. 
To obtain a contradiction, assume $\pi_{1},\pi_{2},\pi,$ and $b$ as in Case(i). 
\item We show that $P_{3}$ is a support of $x$. Let $\tau\in $ fix$_{\mathcal{G}}(P_{3})$. Let $\sigma\in$ fix$_{\mathcal{G}}(P_{2})$ be such that $\sigma(a)=a$, and $\sigma(b)=\tau(a)$. Working in much the same way
as in Case(i), we can see that $\tau(x)=x$.

\item Since $\{a,d\}\subset P_{d}\cup\{a\}\in P_{3}$ we have $(a,d)\in$ fix$_{\mathcal{G}}(P_{3})$. Thus $(a,d)x=x$ as $P_{3}$ is a support of $x$, which is a contradiction. Thus, $f$ is injective.
\end{enumerate}
In both cases, $dom(f)$ is an infinite subset of $A$, and $rng(f)$ is a subset of $X$. 
\end{proof}

\begin{lem}\label{Lemma 6.3}
{\em 
The following hold:
\begin{enumerate}
    \item The set $A$ of atoms is Dedekind finite in $\mathcal{V}_{fp}$.
    \item The set $A$ of atoms has no infinite amorphous subset in $\mathcal{V}_{fp}$.
\end{enumerate}
}
\end{lem}

\begin{proof}
This follows by the arguments of \cite[Propositions 4.3, 4.8]{Bru2016}.
\end{proof}

\begin{lem}\label{Lemma 6.4}
{\em Let $S$ be a non well-orderable set in $\mathcal{V}_{fp}$ with support $E_1$. 
Then, there exists an $x \in S$ with support $E_{2}$ which is a refinement of $E_1$ and an infinite block $P$ in $E_1$ such that the following holds:
\begin{enumerate}
    \item There exists $P_1, P_{2}\subseteq P$ where $P_1,P_2 \in E_2$ and $P_2$ is infinite.
    \item There exists $a \in P_1$ and $b \in P_2$ such that $(a, b) x \not = x$.
\end{enumerate}
}
\end{lem}

\begin{proof}
Since $S$ is not well-ordered, and $E_{1}$ is a support of $S$, there
is a $x \in S$ and a $\phi \in $fix $_{\mathcal{G}}(E_{1})$ such that $\phi(x)\neq x$. 
Let $E_{2}$ be a support of $x$ which is a refinement of $E_{1}$ such that $\vert E_{2}\vert -\vert E_{1}\vert$ is least.
Then there exists an infinite block $P\in E_1$ such that $P=\bigcup_{1\leq i\leq n}P_{i}$ for some $n\geq 2$ where $P_{i}\in E_{2}$ for each $1\leq i\leq n$.
At least one of these $P_{i}$'s must be infinite, say $P_2$. 
Let $E_{3}$ be a partition whose blocks are $(P_{1}\cup P_{2})$, $P_{3}$, ..., $P_{n}$ and all blocks of $E_{2} \backslash \{P_{1},P_{2},...P_{n}\}$. 
Clearly, there exists a $\psi\in$ fix$_{\mathcal{G}}(E_{3})$ such that $\psi(x)\neq x$. Otherwise, $E_{3}$ is a support for $x$ which contradicts the assumption that $E_{2}$ is a support of $x$ such that $\vert E_{2}\vert -\vert E_{1}\vert$ is least.
Since $\psi$ moves finitely many atoms, it is decomposable into transpositions.  Let $\psi=\prod_{1\leq i\leq n} f_{i}$ where $f_{i}$ is a transposition for each $1\leq i\leq n$. Thus there exists at least one $f_{i}$ say $(a,b)$, such that $(a,b)x\neq x$. Moreover, $a,b\in P_{1}\cup P_{2}$, such that $a$ and $b$ both cannot come from either $P_{1}$ or $P_{2}$ (otherwise $(a,b)x=x$ since $E_{2}$ is a support of $x$). Thus one of $a,b$ belongs to $P_{1}$ and the other belongs to $P_{2}$. Without loss of generality, we assume that $a\in P_{1}$ and $b\in P_{2}$. 
\end{proof}

\begin{thm}\label{Theorem 6.5}
{\em Fix any $n,m\in \omega\backslash\{0,1\}$. In $\mathcal{V}_{fp}$, the following hold:

\begin{enumerate}
\item $\mathsf{CS}$ and $\mathsf{A}$.
\item $\mathsf{RRT}^{m}_{n}$ and $\mathsf{RT}^{m}_{n}$.
\item There are no amorphous subsets.\footnote{The proof of this assertion is due to the second and the third authors.}
\item $\mathsf{EDM}$.
\end{enumerate}
}    
\end{thm}

\begin{proof}
(1). This follows by Lemma \ref{Lemma 2.5} since any $\phi\in\mathcal{G}$ moves only finitely many atoms.

(2). Let $X \in \mathcal{V}_{fp}$ be an infinite set. In view of Fact \ref{Fact 3.1}(2,3), we may assume,
without loss of generality, that $X$ is not well-orderable. By Lemma \ref{Lemma 6.2}, $X$ contains a copy of an infinite subset of $A$, say $A'$. Let $\chi:[A']^m \rightarrow C$ be an $n$-bounded coloring in $\mathcal{V}_{fp}$ with support $P$. Since $A'$ is infinite and $P$ is a finite partition of $A$,
there is a $p \in P$ such that $p\cap A'$ is infinite. By the arguments in the proof of Theorem \ref{Theorem 4.1}(1), the (infinite) subset $p\cap A'$ of $A$ is polychromatic for $\chi$, and thus $\mathsf{RRT}^{m}_{n}$ holds for $A'$. Thus, $\mathsf{RRT}^{m}_{n}$ holds for $X$ by Fact \ref{Fact 3.1}(3).

Let $c : [A']^{m}\rightarrow n$ be a coloring in $\mathcal{V}_{fp}$ and $P$ be a support for $c$. Then there exists $p\in P$ such that $p\cap A'$ is infinite. We can see that $p\cap A'$ is monochromatic for $c$. 
Let $x,y\in [p\cap A']^m$.
Consider a permutation $\phi$ of $A$ such that $\phi[x]=y$, $\phi[y]=x$ and $\phi$ fixes all the other atoms. Clearly, $\phi\in \mathcal{G}$ (as $\phi$ moves only finitely many atoms) and $\phi \in$ fix$_{\mathcal{G}}(P)$. Thus we have
that $\phi(c) = c$ as $P$ is a support for $c$. 
In particular, if $c(x) = i$ then we have that
$(y, i) = (\phi[x], \phi(i)) = \phi((x, i)) \in \phi(c) = c$,
which shows that $c(y) = i$ as well. Since $y \in [p\cap A']^{m}$ was arbitrary, we conclude that the set $[p\cap A']^{m}$ is
monochromatic in color $i$. Thus $\mathsf{RT}^{m}_{n}$ holds for $A'$ and so $\mathsf{RT}^{m}_{n}$ holds for $X$.

(3).
Let $S$ be an infinite set in $\mathcal{V}_{fp}$. If $S$ is well-orderable, then it is not amorphous. Suppose $S$ is a non-well-orderable set. We show that $S$ is not amorphous.
Let $E_1$ be a support for $S$. 
By Lemma \ref{Lemma 6.4}, there is an $x \in S$ with support $E_{2}$, which is a refinement of $E_1$, and an
infinite block $P$ in $E_1$ such that the following hold:

\begin{enumerate}
    \item There exists $P_1, P_{2}\subseteq P$ where $P_1,P_2 \in E_2$ and $P_2$ is infinite.
    \item There exists $a\in P_{1},b\in P_{2}$, such that $(a,b)x\neq x$.
\end{enumerate}

\begin{claim}\label{claim 6.6}
{\em 
The following holds:
\begin{enumerate}[(a).]
\item $(a,b_i) x \not = (a,b) x$ for any $b_i \in P_2\backslash \{b\}$. 

\item $(a,b_i) x \not = x$ for any $b_i \in P_2\backslash \{b\}$. 

\item $(a,b_i)x \not = (a,b_j) x$ for any $b_i, b_j \in P_2$ such that $b_{i}\neq b_{j}$.
\end{enumerate} 
}
\end{claim}

\begin{proof}
$(a)$.  Suppose $(a,b_{i}) x = (a,b) x$ for some $b_i \in P_2\backslash \{b\}$. Then, 
\begin{center}
    $(a,b)(b_i,a)(a,b_i) x = (a,b) (b_i,a) (a,b) x \implies (a,b) x = (b,b_i) x $.
\end{center} 
However, $(b,b_i) \in$ fix$_{\mathcal{G}}(E_2)$ and so $(b,b_i) x = x$, while $(a,b) x \not = x$ which is a contradiction.

$(b)$. Suppose $(a, b_i) x = x$ for some $b_i \in P_2\backslash \{b\}$. Thus $(a,b_i)(b, b_i) (a, b_i)x=(a, b_i)(b, b_i) x$. 
Since $(b, b_i) \in$ fix$_{\mathcal{G}}(E_2)$ and so $(b, b_i)x=x$, we have $(a, b) x = (a, b_i) x $, which contradicts $(a)$.

$(c)$. Assume $(a, b_i)x = (a, b_j) x $ for some $b_i, b_j \in P_2$ such that $b_{i}\neq b_{j}$.
Thus, $(a, b_i)(a, b_j)(a, b_i) x =  (a, b_i)(a, b_j)(a, b_j) x$. 
Since $(b_i, b_j) \in$ fix$_{\mathcal{G}}(E_2)$ and so $(b_i, b_j)x=x$, we have 
$(b_i, b_j) x = (a, b_i) x \implies x = (a, b_i)x$     
which contradicts $(b)$.
\end{proof}

We note that for any $b_i \in P_2$, $\{a, b_i\} \subseteq P$. Thus $(a, b_i) \in$ fix$_{\mathcal{G}}(E_1)$. Since $E_1$ is a support of $S$, we have $(a, b_i)S=S$, and so $(a, b_i)y \in S$ for any $y\in S$. Thus,

\begin{center}
    $T= \{(a,b_{i})x : b_{i}\in P_{2}\}$ 
\end{center}
is an infinite subset of $S$ by claim \ref{claim 6.6}$(c)$.
Since the set of atoms $A$ has no infinite amorphous subset in $\mathcal{V}_{fp}$ by Lemma \ref{Lemma 6.3}(2), and $P_{2}$ is an infinite subset of $A$, we can pick a partition of $P_{2}$ into two infinite sets, say $P_{2,1}$ and $P_{2,2}$. Define,

\begin{center}
$U = \{(a, d_i)x : d_i \in P_{2,1}\}$, and
$V = \{(a, e_i)x: e_i \in P_{2,2}\}$.   
\end{center}

Let $E_{3}$ be a partition whose blocks are $P_1 \backslash \{a\}$, $P_{2,1} \cup \{a\}$, $P_{2,2}$, and all blocks of $E_2 \backslash \{P_1, P_2\}$.
Define $Y = \{\pi (x) : \pi \in$ fix$_{\mathcal{G}}(E_3)\}$.
By claim \ref{claim 6.6}$(c)$, $U$ and $V$ are infinite sets. Since $U\subseteq Y$ and $V\subseteq S\backslash Y$, we have $Y$ and $S\backslash Y$ are infinite. Moreover, $E_{3}$ is a support of $Y$ and $S$, as $E_{1}$ is a support of $S$ and $E_{3}$ is a refinement of $E_{1}$. Thus, $E_{3}$ is a support of $S\backslash Y$.
Consequently, $Y$ and $S\backslash Y$ are infinite sets in $\mathcal{V}_{fp}$, and $Y$ and $S\backslash Y$ form a partition of $S$.

(4). In $\mathcal{V}_{fp}$, let $G=(V_{G}, E_{G})$ be a graph and $f:[V_{G}]^{2}\rightarrow \{0,1\}$ be a coloring such that all sets monochromatic in color $0$ are countable, and all sets monochromatic in color $1$ are finite. If $V_{G}$ is well-orderable then we are done by Fact \ref{Fact 3.1}(12). Otherwise, by Lemma \ref{Lemma 6.2}, there exists a bijection from an infinite subset $A'$ of $A$ onto some $H\subset V_{G}$.
Define the following partition of $[H]^{2}$:
$X=\{\{a,b\}\in [H]^{2}: f(\{a,b\})=0\}$, $Y=\{\{a,b\}\in [H]^{2}: f(\{a,b\})=1\}$.    
Since $(H, E_{G}\restriction H)$ is an infinite graph where all sets monochromatic in color $1$ are finite, there is no infinite subset $B' \subseteq H$ such that $[B']^2 \subseteq Y$. Since $\mathsf{RT}^{2}_{2}$ holds in $\mathcal{V}_{fp}$, there is an infinite subset $B \subseteq H$ such that $[B]^2 \subseteq X$. So $(H, E_{G}\restriction H)$ has an infinite set monochromatic in color $0$, say $C$.
By assumption, $C$ is a countably infinite subset of $H$. Thus $A'$ is Dedekind-infinite in $\mathcal{V}_{fp}$ since $\vert H\vert=\vert A'\vert$. Consequently, the set $A$ of atoms is Dedekind-infinite in $\mathcal{V}_{fp}$, which contradicts the fact that $A$ is a Dedekind-finite set in $\mathcal{V}_{fp}$ (see Lemma \ref{Lemma 6.3}(1)).
\end{proof}
\section{Concluding Remarks}
\begin{remark}\label{Remark 7.1}
In \cite[a part of Question 6.3]{Tac2022}, it is asked whether $\mathsf{AC^{LO}}$ implies $\mathsf{CAC}^{\aleph_{0}}_{1}$ in $\mathsf{ZFA}$. We show that {\em $\mathsf{AC^{LO}}$ does not imply $\mathsf{CAC}^{\aleph_{0}}_{1}$ in $\mathsf{ZFA}$}. 
Recently, Tachtsis \cite[Theorem 3]{Tac2024} constructed the following permutation model to prove that $\mathsf{AC^{LO}}$ does not imply $\mathsf{EDM}$ in $\mathsf{ZFA}$:  
We start with a model $M$ of $\mathsf{ZFA} + \mathsf{AC}$ with an $\aleph_{1}$-sized set $A$ of atoms which is a disjoint union of $\aleph_{1}$ unordered pairs, so that $A =\bigcup\{A_{i}: i < \aleph_{1}\}$, $\vert A_{i}\vert = 2$ for all $i < \aleph_{1}$, and $A_{i} \cap A_{j} = \emptyset$ for all $i, j < \aleph_{1}$ with $i \neq j$. Let $\mathcal{G}$ be the group of all permutations $\phi$ of $A$ such that $(\forall i < \aleph_{1})(\exists j < \aleph_{1})(\phi(A_{i}) = A_{j})$. Let $\mathcal{F}$ be the normal filter of subgroups of $\mathcal{G}$ generated by the pointwise stabilizers fix$_{G}(E)$, where $E =\bigcup \{A_{i}:i\in I\}$ for some $I \in [\aleph_{1}]^{\leq\aleph_{0}}$. Let $\mathcal{N}$ be the permutation model determined by $M$, $\mathcal{G}$ and $\mathcal{F}$. In $\mathcal{N}$, the following holds (cf. \cite[claims in the proof of Theorem 3]{Tac2024}):
\begin{enumerate}
\item $\mathsf{AC^{LO}}$ holds.
\item The power set of $\mathcal{A} = \{A_{i} : i < \aleph_{1}\}$ consists exactly of the countable
and the co-countable subsets of $\mathcal{A}$.
\item No co-countable subset of $A$ has a choice function in $\mathcal{N}$.
\end{enumerate}

Modifying the proof of \cite[claim 3 of Theorem 3]{Tac2024}), one can see that $\mathsf{CAC}^{\aleph_{0}}_{1}$ fails in $\mathcal{N}$. Deﬁne a binary relation $\leq$ on the uncountable set $A =\bigcup \mathcal{A}$ as follows: for all $a,b \in A$, let $a\leq b$ if and only if $a = b$ or $a \in A_{n}$, $b \in A_{m}$ and $n < m$. It is easy to see that $\leq$ is a partial order on $A$. If $C \subset A$ is an antichain in $(A,\leq)$, then $C \subseteq A_{i}$ for some $i \in \aleph_{1}$. Thus, all antichains in $(A,\leq)$ are finite as $A_{i}$ is finite for all $i\in \aleph_{1}$. Since any two elements of $A$ are $\leq$-comparable if and only if they belong to distinct $A_{i}$'s, we can see that no chain in $(A,\leq)$ is uncountable by (2) and (3).
\end{remark}

\begin{remark}\label{Remark 7.2}
In \cite[a part of Question 6.4]{Tac2022}, it is asked whether $\mathsf{WOAM}$ implies Kurepa's theorem ($\mathsf{CAC}^{\aleph_{0}}_{1}$) in $\mathsf{ZF}$. The {\em answer to the above question is in the affirmative}. Firstly,
we recall the following known facts:
\begin{enumerate}
    \item The statement ``Every infinite poset $(Q, \leq)$ such that $Q$ is well-orderable has either an infinite chain or an infinite anti-chain'' holds in $\mathsf{ZF}$ (see \cite[Proof of Claim 5]{Tac2016}).
    \item $\mathsf{WOAM}$ implies $\mathsf{CAC}^{\aleph_{0}}$ in $\mathsf{ZF}$ (see \cite[Theorem 6]{Tac2024}).
    \item $\mathsf{WOAM} + \mathsf{CAC}$ implies $\mathsf{CAC}^{\aleph_{0}}_{1}$ in $\mathsf{ZF}$ (see \cite[Theorem 8(1)]{Tac2022}).
\end{enumerate}
We note that $\mathsf{CAC^{\aleph_{0}}}$ implies $\mathsf{CAC}$ in $\mathsf{ZF}$. Let $(P, \leq)$ be an infinite poset. If all chains and all antichains in $P$ are finite, then by $\mathsf{CAC^{\aleph_{0}}}$, $P$ is countably infinite (and hence well-orderable). By (1), $(P, \leq)$ has either an infinite chain or an infinite antichain, contradicting our hypothesis on $P$. 
Thus, by (2), $\mathsf{WOAM}$ implies $\mathsf{CAC}$ in $\mathsf{ZF}$ and so $\mathsf{WOAM}$ implies $\mathsf{CAC}^{\aleph_{0}}_{1}$ in $\mathsf{ZF}$ by applying (3).
\end{remark}

\begin{remark}\label{Remark 7.3} We remark that {\em if $X\in \{\mathsf{CAC^{\aleph_{0}}}, \mathsf{CAC_{1}^{\aleph_{0}}}, \mathsf{DT}, \mathsf{CS}, \mathsf{A},\mathsf{WOAM}, \mathsf{MC}$, ``There are no amorphous sets'', $\mathsf{HT}\}$, then $X$ and $\mathsf{RRT}^{m}_{n}$ are mutually independent in $\mathsf{ZFA}$.}
This follows by the following known facts and the consequences of Theorem \ref{Theorem 4.1}(1,2):
\begin{enumerate}

\item Tachtsis constructed a permutation model $\mathcal{N}$ in the proof of \cite[Theorem 2.1]{Tac2016} where $\mathsf{AC_{2}^{-}}$ fails and the principle ``Every infinite poset has an infinite chain or an infinite antichain'' holds. Thus $\mathsf{RRT}^{m}_{n}$ fails in $\mathcal{N}$ by Proposition \ref{Proposition 3.3}(1).
\item Tachtsis proved that $\mathsf{WOAM}$ holds in $\mathcal{N}$ (see \cite[Lemma 2]{Tac2016}) and $\mathsf{DT}$ holds in $\mathcal{N}$ (see \cite{Tac2019}).

\item Banerjee \cite{Ban2023} proved that $\mathsf{CAC_{1}^{\aleph_{0}}}$ holds in $\mathcal{N}$. Furthermore, Banerjee and Gopaulsingh \cite{BG2023} observed that $\mathsf{CS}$ and $\mathsf{A}$ hold in $\mathcal{N}$, where as Banerjee and Gyenis \cite{BG2021} observed that $\mathsf{CAC^{\aleph_{0}}}$ holds in $\mathcal{N}$.

\item Since $\mathsf{DC}$ does not imply  $\mathsf{CAC^{\aleph_{0}}_{1}}$ in $\mathsf{ZF}$ (see \cite[Corollary 4.6]{Ban2023}) and $\mathsf{DC}$ implies $\mathsf{DF=F}$ in $\mathsf{ZF}$, we obtain that $\mathsf{DF=F}$ (and thus $\mathsf{RRT}^{m}_{n}$) does not imply $\mathsf{CAC^{\aleph_{0}}_{1}}$ in $\mathsf{ZF}$ by Proposition \ref{Proposition 3.2}(2).

\item In \cite[the proof of Theorem 4.1(4)]{BG2023}, Banerjee and Gopaulsingh proved that $\mathsf{DF=F}$ does not imply $\mathsf{CAC^{\aleph_{0}}}$ in $\mathsf{ZF}$. By Proposition \ref{Proposition 3.2}(2), $\mathsf{RRT}^{m}_{n}$ does not imply $\mathsf{CAC^{\aleph_{0}}}$ in $\mathsf{ZF}$.

\item In the permutation model $\mathcal{V}$ from \cite[Theorem 9(4)]{Tac2022}, $\mathsf{DT}$ fails and $\mathsf{DC_{\aleph_{1}}}$ holds, and hence $\mathsf{DC}$, and $\mathsf{DF=F}$ hold as well. By Proposition \ref{Proposition 3.2}(2), $\mathsf{RRT}^{m}_{n}$ does not imply $\mathsf{DT}$ in $\mathsf{ZFA}$.

\item In the second Fraenkel model $\mathcal{N}_{2}$, $\mathsf{AC}_{2}^{-}$ fails
whereas $\mathsf{MC}$ and the statement ``There are no amorphous
sets” hold (see \cite{HR1998}). Thus by Proposition \ref{Proposition 3.2}(2), $\mathsf{RRT}^{m}_{n}$ fails in $\mathcal{N}_{2}$. 

\item $\mathsf{HT}$ is true in $\mathcal{N}_{2}$ where $\mathsf{RRT}^{m}_{n}$ fails, since $\mathsf{MC}^{\omega}$ (Every denumerable family of non-empty sets has a multiple choice function) holds in $\mathcal{N}_{2}$ and $\mathsf{MC}^{\omega}$ implies $\mathsf{HT}$ in $\mathsf{ZF}$(see \cite[Proposition 1]{Tac2022a}, \cite{Fer2023}). However, $\mathsf{HT}$ fails in the basic Fraenkel model $\mathcal{N}_{1}$  since $\mathsf{HT}$ implies that there are no amorphous sets in $\mathsf{ZF}$ (see \cite[Theorem 1(5)]{Tac2022a}, \cite{Fer2023}).
Thus, by Theorem \ref{Theorem 4.1}(1), $\mathsf{RRT}^{m}_{n}$ does not imply $\mathsf{HT}$ in $\mathsf{ZFA}$.  
\end{enumerate}
\end{remark}

\section{Questions}

Fix $m,n\in\omega\backslash\{0,1\}$. 
\begin{question}
Does there exist a model of $\mathsf{ZF}$ where $\mathsf{RT}^{m}_{n}$ holds but $\mathsf{RRT}^{m}_{n}$ fails?   
\end{question}

\begin{question}
Does the Boolean Prime Ideal theorem  imply $\mathsf{RRT}^{m}_{n}$ in $\mathsf{ZF}$?   
\end{question}

\begin{question}
Does $\mathsf{RRT}^{m}_{n}$ imply K\H{o}nig's lemma in $\mathsf{ZF}$ or in $\mathsf{ZFA}$?  
\end{question}

\begin{question}
Does $\mathsf{EDM}$ imply $\mathsf{RRT}^{2}_{2}$ in $\mathsf{ZF}$ or in $\mathsf{ZFA}$?  
\end{question}

\begin{question}
Does $\mathsf{CRT}$ hold in the basic Fraenkel model $\mathcal{N}_{1}$?    
\end{question}


\begin{thebibliography}{99}
\bibitem{Ban2023} Banerjee A.,
\newblock {\em Maximal independent sets, variants of chain/antichain principle and cofinal subsets without $\mathsf{AC}$},
\newblock {Comment. Math. Univ. Carolin.}
\newblock \textbf{64}
\newblock (2023),
\newblock no. 2,
\newblock 137-159.

\newblock DOI: \url{https://doi.org/10.14712/1213-7243.2023.028}.


\bibitem{BG2023}  Banerjee A., Gopaulsingh A.,
\newblock {\em On the Erd\H{o}s--Dushnik--Miller theorem without $\mathsf{AC}$},
\newblock Bull.
Pol. Acad. Sci. Math.
\newblock \textbf{71}
\newblock (2023),
\newblock 1–21.

\newblock DOI: \url{https://doi.org/10.4064/ba221221-6-6}.

\bibitem{BG2021} Banerjee A., Gyenis Z.,
\newblock {\em Chromatic number of the product of graphs, graph homomorphisms, antichains and cofinal subsets of posets without {\rm AC}},
\newblock {Comment. Math. Univ. Carolin.}
\newblock \textbf{62} 
\newblock (2021),
\newblock no. 3,
\newblock 361-382.

\newblock DOI: \url{https://doi.org/10.14712/1213-7243.2021.028}.

\bibitem{Bla1977} Blass A., \newblock {\em Ramsey’s theorem in the hierarchy of choice principles}, \newblock {J. Symb. Log.}
\newblock \textbf{42} 
\newblock (1977),
\newblock 387–390.

\newblock DOI: \url{https://doi.org/10.2307/2272866}.

\bibitem{Bru2016}
Bruce B.B.,
\newblock {\em A Permutation Model with Finite Partitions of the Set of Atoms as
Supports}, 
\newblock {Rose-Hulman Undergraduate Mathematics Journal}
\newblock \textbf{17} 
\newblock (2016),
\newblock no. 1,
\newblock 45-59.

\newblock url: \url{https://scholar.rose-hulman.edu/cgi/viewcontent.cgi?article=1003&context=rhumj}.



\bibitem{ER1950}
Erd\H{o}s P., Rado R., 
\newblock {\em A combinatorial theorem}, 
\newblock {J. London Math. Soc.}
\newblock \textbf{25}
\newblock (1950),
\newblock no. 4,
\newblock 249-255. 

\newblock DOI: \url{ https://doi.org/10.1112/jlms/s1-25.4.249}.

\bibitem{Fer2023}
Fern\'{a}ndez-Bret\'{o}n D., 
\newblock {\em Hindman’s theorem in the hierarchy of choice principles}, 
\newblock {J. Math. Log.}
\newblock (2023).

\newblock DOI: \url{https://doi.org/10.1142/S0219061323500022}.

\bibitem{FT2007}
Forster T. E., Truss J. K., 
\newblock {\em Ramsey’s theorem and K\H{o}nig’s Lemma}, 
\newblock {Arch. Math. Logic}
\newblock \textbf{46}
\newblock (2007),
\newblock 37– 42. 

\newblock DOI: \url{https://doi.org/10.1007/s00153-006-0025-z}.

\bibitem{GRS1980}
Graham R.L., Rothschild B. L., and Spencer J. H.,
\newblock {\em Ramsey Theory}, 
\newblock {second ed. Wiley-Interscience Series in Discrete Mathematics and Optimization.
John Wiley $\&$ Sons Inc., New York},
\newblock (1990). 

\bibitem{HR1998}
Howard P., Rubin J. E.,
\newblock {\em Consequences of the Axiom of Choice},
\newblock Mathematical Surveys and
Monographs, 59, American Mathematical Society, Providence,
\newblock 1998.

\newblock DOI: \url{http://dx.doi.org/10.1090/surv/059}.

\bibitem{HS1993}
Howard P., Solski J.,
\newblock {\em The strength of the $\Delta$-system lemma},
\newblock Notre Dame J. Form. Log.
\newblock \textbf{34}
\newblock (1993)
\newblock no. 1,
\newblock 100–106.

\newblock DOI: \url{https://doi.org/10.1305/ndjfl/1093634567}.

\bibitem{HST2016}
Howard P., Saveliev D. I., Tachtsis E., 
\newblock {\em On the set-theoretic strength of the existence of disjoint cofinal sets in posets without maximal elements},
\newblock {MLQ Math. Log. Q.}
\newblock \textbf{62} 
\newblock (2016),
\newblock no. 3,
\newblock 155-176.

\newblock DOI: \url{https://doi.org/10.1002/malq.201400089}.

\bibitem{Jec1973}
Jech T. J.,
\newblock {\em The Axiom of Choice},
\newblock Stud. Logic Found. Math., 75, North-Holland Publishing
Co., Amsterdam, American Elsevier Publishing, New York,
\newblock 1973.

\bibitem{Kle1969}
Kleinberg E.M.,
\newblock {\em The independence of Ramsey’s theorem},
\newblock J. Symb. Log.
\newblock \textbf{34}
\newblock (1969), 
\newblock no. 2,
\newblock 205–206,

\newblock DOI: \url{ https://doi.org/10.2307/2271095}.

\bibitem{Pal2013}
Palumbo J.,
\newblock {\em Comparisons of polychromatic and monochromatic Ramsey theory},
\newblock J. Symb. Log.
\newblock \textbf{78}
\newblock (2013),
\newblock no. 3, 
\newblock 951–968.

\newblock DOI: \url{https://doi.org/10.2178/jsl.7803130}.

\bibitem{Tac2024}
Tachtsis E.,
\newblock {\em On the deductive strength of the Erd\H{o}s-Dushnik-Miller
theorem and two order-theoretic principles},
\newblock {Monatsh. Math.}
\newblock (2024).

\newblock DOI: \url{https://doi.org/10.1007/s00605-023-01933-z}.


\bibitem{Tac2022}
Tachtsis E.,
\newblock {\em On a theorem of Kurepa for partially ordered sets and weak choice},
\newblock {Monatsh. Math.}
\newblock \textbf{199}
\newblock (2022),
\newblock no. 3,
\newblock 645–669.

\newblock DOI: \url{https://doi.org/10.1007/s00605-022-01751-9}.

\bibitem{Tac2022a}
Tachtsis E.,
\newblock {\em Hindman’s theorem and choice},
\newblock {Acta Math. Hungar.}
\newblock \textbf{168}
\newblock (2022),
\newblock no. 3,
\newblock 402–424.

\newblock DOI: \url{https://doi.org/10.1007/s10474-022-01288-1}.

\bibitem{Tac2021}
Tachtsis E.,
\newblock {\em $\mathsf{MA(\aleph_{0})}$ restricted to complete Boolean algebras and choice},
\newblock {MLQ Math. Log. Q.}
\newblock \textbf{67} 
\newblock (2021),
\newblock no. 4,
\newblock 420-431.

\newblock DOI: \url{ https://doi.org/10.1002/malq.202000031}.

\bibitem{Tac2019}
Tachtsis E.,
\newblock {\em Dilworth's decomposition theorem for posets in $\mathsf{ZF}$},
\newblock {Acta Math. Hungar.}
\newblock \textbf{159}
\newblock (2019),
\newblock no. 2,
\newblock 603-617.

\newblock DOI: \url{https://doi.org/10.1007/s10474-019-00967-w}.

\bibitem{Tac2019b}
E. Tachtsis,
\newblock {\em On the existence of almost disjoint and MAD families without $\mathsf{AC}$},
\newblock {Bull.
Polish Acad. Sci. Math.}
\newblock \textbf{67}
\newblock (2019),
\newblock no. 2,
\newblock 101-124.

\newblock DOI: \url{ https://doi.org/10.4064/ba8148-3-2019}.

\bibitem{Tac2018}
Tachtsis E.,
\newblock {\em On the Minimal Cover Property and Certain Notions of Finite},
\newblock {Arch. Math. Logic} 
\newblock \textbf{57} 
\newblock (2018),
\newblock no. 5-6,
\newblock 665-686.

\newblock DOI: \url{https://doi.org/10.1007/s00153-017-0595-y}.

\bibitem{Tac2016}
Tachtsis E.,
\newblock {\em On Ramsey's Theorem and the existence of Infinite Chains or Infinite Anti-Chains in Infinite Posets},
\newblock {J. Symbolic Logic}
\newblock \textbf{81}
\newblock (2016),
\newblock no. 1,
\newblock 384-394.

\newblock DOI: \url{https://doi.org/10.1017/jsl.2015.47}.

\bibitem{Tru1989}
Truss J. K.,
\newblock {\em Infinite permutation groups II. Subgroups of small index},
\newblock {J. Algebra}
\newblock \textbf{120}
\newblock (1989),
\newblock 494–515.

\newblock DOI:
\url{https://doi.org/10.1016/0021-8693(89)90212-3}.
\end{thebibliography}
\end{document}